\documentclass[final]{siamltex}
\usepackage[ansinew,latin1]{inputenc}
\usepackage[english]{babel}
\usepackage{amsfonts}      
\usepackage{amsmath}
\usepackage{latexsym}
\usepackage{verbatim}    
\usepackage{enumerate}
\usepackage{epsfig} 
\usepackage[all]{xy}
\def\HAMS{H^{(s)}}
\def\HAML{H^{(l)}}
\def\BARHAMS{\bar{H}^{(s)}}
\def\BARHAML{\bar{H}^{(l)}}
\def\WHAML{\widetilde{H}^{(l)}}

\def\VL{V^{(l)}}
\def\US{U^{(s)}}
\def\UL{U^{(l)}}
\def\BARUL{\bar{U}^{(l)}}

\def\AONE{\mathrm{(A1)}}
\def\ATWO{\mathrm{(A2)}}
\def\CPU{\mathrm{CPU}}

\def\rej{\mathrm{rej}}
\def\multi{\mathrm{multi}}
\def\SUCC{\mathrm{succ}}

\def\NULL{{\mathrm{null}}}

\def\BARIT#1{{\bar {#1}}}

\def\BARHZ{\BARIT{H}^{(0)}}

\def\BARMZ{\BARIT{\mu}^{(0)}}

\def\LATT{{\Lambda}}
\def\LATTN{{\Lambda}_N}

\def\LATTC{{\bar{\Lambda}_{M}}}
 
\def\COP{\mathbf{T}}

\def\RELENT#1#2{\Rr\left(#1|#2\right)}

\def\LOPER{\mathcal{L}}
 
\def\Es{\mathbb{E}_{\sigma_0}}
\def\HS{\widetilde{\sigma}}
\def\HU{\widetilde{U}}

\def\st{\{\sigma_t\}_{t\ge 0}}
\def\ast{\{\widetilde\sigma_t\}_{t\ge 0}}
\def\etat{\{\eta_t\}_{t\ge 0}}

\def\StT{\{\sigma_t\}_{t \in [0,T]}}

\def\aStT{\{\tilde \sigma_t\}_{t \in [0,T]}}
\def\muN{\mu_{N,\beta}}
\def\mutN{\widetilde\mu_{N,\beta}}

\def\HC{\widetilde{c}}
\def\HL{\widetilde{\lambda}}
\def\HLl{\widetilde{\Ll}}
\def\HQm{\widetilde {\Dd}_{[0,T]}}
\def\Qm{ {\Dd}_{[0,T]}}
\def\crf{c_{\mathrm{rf}}}
\def\trf{\mathrm{rf}}

\def\Csp{\{\sigma': \COP\sigma'=\eta'\}}

\def\CMspace{\BARIT\Sigma_M}

\def\Linf{L^{\infty}}
\def\BIGO{\mathcal{O}}
\def\VIZ#1{(\ref{#1})}

\def\SEP{\,|\,}

\def\EXPECT#1{\E\left[{#1}\right]}

\def\Dd{\mathcal{D}}

\def\Ll{\mathcal{L}}

\def\Rr{\mathcal{R}}

\def\Oo{\mathcal{O}}

\def\Hh{\mathcal{H}}

\def\E{\mathbb{E}}

\def\R{\mathbb{R}}

\def\EXP#1{e^{#1}}

\def\EXPECT{{\mathbb{E}}}

\def\PROB#1{{\mathbb{P}\left({#1}\right)}}

\def\COMMA{\,,}             
\def\PERIOD{\,.}            
\def\SEP{{\,|\,}}           

\def\VIZ#1{(\ref{#1})}      

\def\BIGO{\Oo}

\newtheorem{algorithm}{Algorithm}
\newtheorem{method}{Method}
\newtheorem{remark}{Remark}[section]

\newtheorem{prop}{Proposition}

\title{Spatial multi-level 
interacting particle simulations and information theory-based error 
       quantification.\thanks{The research of  E.K. was supported by the National Science Foundation under the grant
                              NSF-CMMI-0835582 and the Department of Energy under the grant DE-SC000233; 
                              M.A.K. was partially supported by  the grant NSF-CMMI-0835673;  
                              P.P. was partially supported by the grant NSF-DMS-0813893.
                              }}
               \author{Evangelia Kalligiannaki\footnotemark[3]
                \and Markos A. Katsoulakis \footnotemark[3]
                \and  
                Petr Plech\'a\v{c}\footnotemark[2]
}
\begin{document}

\maketitle
\renewcommand{\thefootnote}{\fnsymbol{footnote}}
\footnotetext[2]{
           Department of Mathematical Sciences,
           University of Delaware,
           Newark, DE 19716, USA 
              ( {\tt plechac@math.udel.edu}). }
\footnotetext[3]{
		Department of Mathematics and Statistics, University
                of Massachusetts,  Amherst, MA 01003, USA
	         ({\tt ekalligi@math.umass.edu}, {\tt  markos@math.umass.edu}). 
                }  
\begin{abstract} 
We propose a hierarchy of multi-level
kinetic Monte Carlo
methods for  sampling high-dimensional, stochastic lattice particle dynamics
with complex interactions. The method is based on the efficient
coupling of different spatial resolution levels, taking advantage of the low sampling
cost in a coarse space and by developing local reconstruction strategies from coarse-grained dynamics. 
Microscopic reconstruction %
corrects possibly significant errors introduced through coarse-graining,  leading to the %
controlled-error approximation of the sampled stochastic process. 
In this manner, the proposed multi-level algorithm overcomes known shortcomings of coarse-graining
of particle systems with complex interactions such as combined long and short-range particle
interactions and/or complex lattice geometries.
Specifically, we provide error analysis   for the approximation of long-time stationary dynamics in terms
of relative entropy 
and prove that information loss in the multi-level methods  is growing linearly in time, which in turn implies 
that an appropriate observable in the stationary regime  is the information loss 
of the path measures per unit time. We show that this observable  can be either estimated {\em a priori},  
or it can be tracked computationally {\em a posteriori} in the course of a simulation.
The stationary regime is of critical importance to molecular simulations as it is  relevant to  
long-time sampling, obtaining phase diagrams  and in studying metastability properties of high-dimensional complex systems.
Finally, the multi-level nature of the method 
provides  flexibility  in combining rejection-free and null-event implementations,  generating  a hierarchy  of   algorithms
 with an adjustable number of rejections that includes  well-known rejection-free and null-event algorithms.
\end{abstract}
\begin{keywords} 
kinetic Monte Carlo, coarse graining, multiple scales, phase transition, information theory, multi-level methods, 
relative entropy, error analysis.
\end{keywords}
\begin{AMS}
 65C05, 65C20, 82C22, 82C20
\end{AMS}
\pagestyle{myheadings}
\markboth{E. Kalligiannaki, M. A. Katsoulakis and P. Plech\'a\v{c}}{Multi-level kinetic Monte Carlo}
\section{Introduction}\label{intro}
One of the widely used computational methods at atomistic scales simulating stochastic particle systems 
is the continuous time or kinetic Monte Carlo (kMC) method. 
The first implementations of rejection-free kinetic Monte Carlo methods in molecular simulations
go back to  the stochastic simulation
algorithm (SSA) of Gillespie for well-mixed systems, \cite{GILLESPIE}, and the
n-fold method or the BKL method of Bortz, Kalos and Lebowitz, \cite{BKL}, for spatially
distributed Ising-type systems. Traditionally kMC algorithms are
serial, explicit time-stepping methods, limiting
their applicability, due to the high computational cost per event (time-step),
to moderate size  systems. The principal part of the cost consists of searching for an event
and updating the reaction rates. In the last decade research works have been
focusing on developing sophisticated search and update techniques to reduce this
computational cost,  \cite{ACDV,SPPARKS09,VOTER}. For example, 
search algorithms have been proposed using binary tree
pointers \cite{GibsonBruck, TS08} in order to obtain $\BIGO(\log N)$ complexity in the size
$N$ of the simulated system. The BKL method
is designed to reduce the cost of the searching step by
lumping  the transition states into classes of the same probability. All these techniques even though they
reduce the computational cost per event are highly demanding in computer
storage and implementation overhead. 
An alternative to the rejection-free methods is based on the {\it uniformization} of the simulated
continuous Markov chain.  The implementation leads to null-event algorithms, 
which at each step require calculation of the rate for only one, arbitrarily chosen, 
event that is  accepted or rejected with a probability provided by the uniformization of the process, \cite{ACDV}.
Although null-event methods reduce significantly the computational cost per
Monte Carlo step they can be highly inefficient when the acceptance probability
of an event becomes small, resulting to exceedingly  small time-steps. 
 
In this work we propose a {\em partially} rejection-free kinetic Monte Carlo method,
the multilevel kinetic coarse-grained Monte Carlo (ML-KMC), for  sampling  high-dimensional 
lattice systems with complex interactions and/or lattice geometries. Our primary interest in
simulating extended systems in which the system size $N$ is large and the behavior is governed by
the infinite volume or thermodynamic limit (i.e., regimes where $N\to\infty$). We distinguish between
long-range interaction potentials, whose range $L \sim N^{1/d}$ is comparable to the system dimensions 
$N^{1/d}$, where $d$ is the dimension of the lattice, and short-range potentials which have a fixed
interaction range $S$ independent of the system size $N$. Typically in the lattice-gas simulations
the long-range potentials are Coulomb or Lenard-Jones and short-range ones are represented by 
nearest-neighbor or next nearest-neighbor potentials. 
Interactions that combine long and short-range potentials are difficult to be handled efficiently
by existing BKL-type algorithms since the number of classes grow exponentially with the
length of the interaction range and thus making an implementation intractable, \cite{ACDV}. 
A possible
remedy is to simulate a system with compressed (coarse-grained) potentials that have a shorter
interaction range. We have demonstrated that smooth long-range potentials can be coarse-grained
with controlled errors leading to highly efficient coarse-grained Monte Carlo (CGMC) methods, \cite{KMV, KPS, KPRT, KPRT2}.
On the other hand a simple coarse-graining of singular, long-range as well as short-range potentials does not
lead to an accurate approximation of the renormalized Hamiltonian and results to loss of important microscopic information,
due to the presence of strong short-ranged spatial correlations.
As demonstrated in \cite{KPRT3, AKPR} 
approximations of coarse-grained Hamiltonians accounting for  short-range or singular potentials necessarily involve multi-body interaction
terms whose implementation can become quickly  computationally expensive. The pitfalls of coarse-graining are
well-known also from molecular simulations of polymeric systems where they produce spurious phase changes or
incorrectly predict existing phase transitions, \cite{doi, KPR}. Similarly, coarse-graining is expected to be challenging in
systems with complex lattice geometries, e.g.  \cite{graphKMC},  that can also induce complicated spatial correlations.

In this paper we demonstrate the use of the proposed ML-KMC method in the efficient and accurate simulation  of such complex systems where coarse-graining either fails or it is computationally very expensive.
We focus on the important example of systems where competing short- and long- ranged interaction forces are present 
and lead to complicated phase diagrams and pattern formation in various physical and chemical systems, \cite{Nature2001,ChVl,LEFK}. 

\section{Overview of the proposed method}
The key ingredient of the proposed method is the
multi-level sampling of the evolution process based on the knowledge of an even less accurate coarse-grained,
meso/macroscopic dynamics. The present study is an extension of \cite{KKP} from the
equilibrium to dynamical sampling sharing the same principle of efficient
coupling of different resolution levels. In \cite{KKP} we proposed a multi-level
Coarse-Grainined Metropolis-Hastings algorithm appropriate for sampling equilibrium
properties of systems. Although the embedded Markov chain generated by the Metropolis algorithm
converges to the correct equilibrium distribution it does not preserve physical dynamics, 
a fact that motivated the present work.
We present the method in a general framework to demonstrate its applicability to on- and off-lattice systems, 
and present in detail the application to stochastic lattice systems with short- and long-range interactions. 

The dynamics of stochastic systems on a countable configuration space $\Sigma$
are determined by a continuous time Markov process $(\st, \Ll) $, with the infinitesimal generator $\Ll: \Linf(\Sigma) \to \Linf(\Sigma) $
defined by the rates $ c(\sigma,\sigma')$, $\sigma,\sigma'\in \Sigma$:
\begin{equation}\label{Gen1}
	\Ll \phi (\sigma) = \sum_{\sigma'\in\Sigma}
	c(\sigma,\sigma')\left(\phi(\sigma')-\phi(\sigma) \right)\COMMA
 \end{equation}
for every observable defined as any $ \phi\in L^{\infty}(\Sigma)$. More specifically,
in kMC we typically  compute expected values of such observables, that is quantities such as
\begin{equation}
   u(\zeta, t):=\EXPECT^\zeta[f(\sigma_t)]=\sum_\sigma f(\sigma)P(\sigma, t; \zeta)\, ,
\label{propagator1}
\end{equation}
conditioned on the initial data $\sigma_0 = \zeta$. 
On the other hand, the evolution of the entire system at any time $t$ is described by the  transition probabilities
$P(\sigma, t ; \zeta):=\PROB{\sigma_{t} = \sigma\SEP \sigma_0 = \zeta}$
where $\zeta \in \Sigma$ is any  initial configuration. The transition probabilities 
satisfy the Forward Kolmogorov Equation ({Master Equation}), \cite{Gardiner04},
\begin{equation}
\partial_t P(\sigma, t ; \zeta)=\sum_{\sigma',
\sigma'\ne \sigma} c(\sigma', \sigma)P(\sigma', t; \zeta)-c(\sigma, \sigma')P(\sigma, t ; \zeta)\, ,
\label{master}
\end{equation}
where $P(\sigma, 0 ; \zeta)=\delta (\sigma-\zeta)$ and $\delta (\sigma-\zeta)=1$ if $\sigma=\zeta$
and zero otherwise.
By a straightforward calculation using \VIZ{master} we obtain that  the observable \VIZ{propagator1}
satisfies the initial value problem
\begin{equation}\label{ODE}
\partial_t u(\zeta, t)=\LOPER u(\zeta, t)\, , \quad\quad u(\zeta, 0)=f(\zeta)\, ,
\end{equation}
The numerical implementation of the evolution of the process is realized with the embedded Markov chain $\{X_n\}_{n\geq 0}$,
$X_n = \sigma_{n\delta t}$ with transition probabilities
\begin{equation}\label{MicroPr}
           p(\sigma,\sigma') =\frac{c(\sigma,\sigma')}{\lambda(\sigma)} \COMMA \;\;\;
           \lambda(\sigma) = \sum_{\sigma'\in \Sigma} c(\sigma,\sigma') \COMMA
\end{equation}
where $p(\sigma,\sigma')$ is the probability of a jump from the state $ \sigma$ to
$\sigma'$. The residence times $\delta t_{\sigma} $ for which the system stays in the state
$\sigma$ before a jump %
 is distributed according to an exponential law with
the parameter $\lambda(\sigma)$. 

\medskip\noindent
{\em Method.} 
The proposed method generates  an  approximate %
process $( \ast,\HLl)$  of the stochastic
process $( \st,\Ll)$. It is based on projecting the
microscopic space $\Sigma$ into a coarse space $\BARIT\Sigma$ with less degrees
of freedom and on the knowledge of a coarse rate function $ \BARIT{c}(\eta,\eta')$
which captures macroscopic information from $c(\sigma,\sigma') $. We denote the
coarse space variables $\eta=\COP\sigma$ defined by a projection
operator $\COP: \Sigma \to \bar{\Sigma}$. 
For example, for stochastic lattice systems that we elaborate on in this work, approximate coarse rate 
functions are explicitly known from coarse graining (CG) techniques  of \cite{KMV,KMV1}. 
In this work we analyze a two-level approach, i.e., coupling two configuration spaces  
$\Sigma$ and $ \bar{\Sigma}$ with different  resolutions, while a multi-level extension 
can be considered analogously. 
The ML-KMC  method consists of the following steps, a schematic description is demonstrated  in
Figure~\ref{scheme}:
\begin{enumerate}[i)]
\item Construct a computationally inexpensive  approximating CG process
      on the coarse space $ \BARIT\Sigma$ described by a coarse generator $ \BARIT
      \Ll$ with rates $\BARIT{c}(\eta, \eta')$. 
\item Define the ``reconstruction" rates $\crf(\sigma'|\eta',\sigma)$
      constrained on the updated coarse state $\eta'$ that are simple to simulate and
      such that
\begin{equation*}
     \BARIT c(\eta,\eta') \crf(\sigma'|\eta',\sigma) \ \textrm{ {\it approximates} } \ c(\sigma,\sigma') \PERIOD
\end{equation*}
\end{enumerate}
The approximation and its error is quantified in
Section~\ref{sectionError}.
The overall procedure can be thought as a
reconstruction in dynamics of stochastic processes from a CG process.
Furthermore, this procedure generates stochastic processes  that are 
{\it controlled error}  
approximations of the process $(\st, \Ll)$, determined by the   reconstruction rates $\crf(\sigma'|\eta',\sigma) $ 
and the level of coarsening.  The function $\crf(\sigma'|\eta',\sigma) $ enriches the CG procedure by re-inserting  details that were smoothed out
by the coarsening procedure.
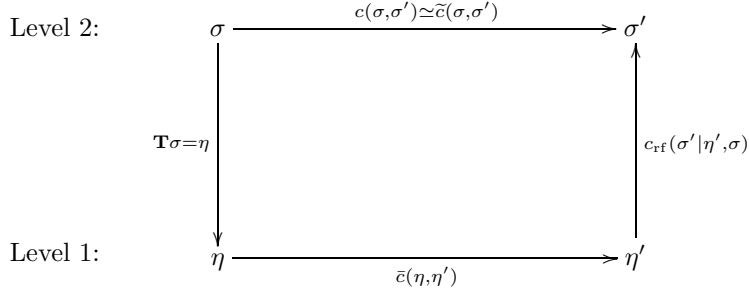
\begin{figure}[t]
 \centerline{
  \xymatrixrowsep{1in}
  \xymatrixcolsep{0.02in}
     \xymatrix{ \text{Level 2:} & \\
     \text{Level 1:} &
  }
  \xymatrixrowsep{1in}
  \xymatrixcolsep{2in}
     \xymatrix{
        \sigma \ar[r]^{c(\sigma,\sigma') \simeq \HC(\sigma,\sigma')}
         \ar[d]_{\COP\sigma=\eta} &
         \sigma' \\
         \eta \ar[r]_{\BARIT{c}(\eta,\eta')} & \eta' \ar[u]_{\crf(\sigma'|
         \eta',\sigma) }  
  }\label{scheme}
 }\caption{Two-level decomposition, compressing with $\COP$ and reconstructing
          with $c_{\trf}$, of the evolution process per event.}
\end{figure}
\medskip\noindent
{\it Implementation}. The multilevel nature of the method provides the flexibility of
combining rejection-free and null-event implementation algorithms at each resolution 
and consists of:
\begin{enumerate}[i)]
   \item  A rejection-free algorithm for sampling in the coarse space with $\BARIT
          c(\eta,\eta')$ where the reduction of the computational cost compared to
          microscopic sampling is significant due to the compression of spatial scales and
          interactions  range. 
          The rejection-free algorithm selects the most probable coarse
          state $\eta'$ that the system will evolve.
   \item A null-event algorithm for sampling at the fine space with
         $c_{\trf}(\sigma'|\eta',\sigma) $, however with a low rejection rate due to the fact
         that $ \eta'$ was chosen as the most probable event.
\end{enumerate}
This partially rejection-free implementation approach suggests a non-constant
time step update in contrast to null-event methods where the time update is uniform for all system states.
Non-constant time step updating algorithms have been proposed in \cite{ABL} designing a class of kinetic Monte Carlo 
algorithms with an adaptive time step  interpolating between BKL and
null-event algorithm. The variable time step in these  algorithms is based on the adaptively improved upper 
bounds of the exponential parameter controlling the time step distribution, while  in our proposed method it is the result of 
the  combination of BKL  and rejection-free implementations on different  resolution state spaces.     
However, one can  further enhance the ML-KMC method with an  adaptive   time step kinetic Monte Carlo in the spirit of \cite{ABL},  in view of the freedom on 
the choice of implementation techniques  in each level.

\medskip\noindent
{\it Analysis.} 
We provide numerical analysis for (a) finite-time weak error estimates  and (b) long-time stationary dynamics for the proposed ML-KMC algorithm.
The primary challenge in the finite-time weak error estimates is to obtain bounds that are independent of  the high-dimension of the interacting particle system,
and this is accomplished by focusing on suitable macroscopic observables.
However, this technique leads to estimates that involve constants that grow 
exponentially in time, as is the case also in many of the classical numerical analysis estimates for stochastic differential equations 
or partial differential equations that rely on Gronwall inequality-type  arguments. In order to overcome this difficulty we develop a different approach that
estimates the error in long-time behavior, i.e.  in stationary regimes for reversible or irreversible  processes,
using the relative entropy of the path measure.
In kMC  as
well as in molecular simulations in general, we are often primarily interested in long-time,
stationary regimes, including (i) the behavior of the stationary measure and its
sampling, as well as (ii) stationary dynamics, i.e. dynamics where the initial
measure is the stationary distribution reached after long-time integration. 
The latter
is an especially important regime describing dynamic transitions between
metastable states in complex energy landscapes, while at this regime we construct the system's phase diagrams.

The error analysis  study we present reveals that the  relevant quantity for assessing 
long-time simulations in the stationary regime  is the {\it information loss per unit time}  in the path space measure.
In fact, the issues related to the error analysis in the stationary dynamics regime is the primary novelty in our paper from a numerical analysis perspective
and they are
not restricted only to the multilevel  kMC algorithms; they are ubiquitous for
numerical approximations of reversible and irreversible stochastic dynamics such as Langevin stochastic
differential equations (SDEs), thermostatted dynamics, dissipative particle dynamics (DPD) methods, etc. 
For such systems there is a wealth of
approximating schemes, the simplest being in the time stepping, e.g., explicit,
implicit, predictor-corrector, operator splitting, etc. 
We expect that  our proposed entropy-based
perspective  could be used to assess such numerical schemes at the stationary
dynamics regime in a quantifiable manner, in a variety of stochastic,  extended as well as finite-dimensional 
systems.
We begin this work in Section~\ref{sectionKMC} presenting the  method  and continue with Section~\ref{implementations} 
proposing  an efficient implementation strategy leading to a partially rejection free method.  Application of the method in stochastic lattice
systems is presented in Section~\ref{stlattice} for adsorption-desorption
Arrhenius dynamics. In Section~\ref{sectionError} we prove
error estimates that provide a quantitative control of the approximating
process. 
Section~\ref{Complexity} establishes computational efficiency of the method over conventional sampling
techniques.   The performance of the method is tested in
Section~\ref{sectionArrheniuskMC-Null} using a benchmark model of Arrhenius
dynamics with a competing short- and long-range interaction potential.

\section{Multilevel  kinetic Monte Carlo}\label{sectionKMC}
The ML-KMC method is a kinetic Monte Carlo method   
generating %
{\it controlled-error approximate} dynamics of the
stochastic process $( \{\sigma_t\}, \Ll)_{t\ge 0}$,  based on the
decomposition of the rate function $c(\sigma,\sigma') $ into  coarse  and corresponding
reconstructing terms, such that
 \begin{equation*} 
  c(\sigma,\sigma') \approx \HC(\sigma,\sigma')= \prod_{i=0}^{I}\BARIT{c}^{(i)}(\eta_i,\eta'_i) \crf^{(i)}(\eta'_{i-1}|\eta'_i,\eta_{i-1})\COMMA
 \end{equation*}
where $\eta_0\equiv \sigma$ and $  \eta_i =\COP_{i-1}\eta_{i-1}, i=1\dots,I  $ are  variables  in a   
hierarchy of coarse spaces with the decreasing numbers of degrees of  freedom.
 Coarsening consists of projecting the microscopic space into a coarse space $\BARIT\Sigma$ with less degrees of freedom
for which a coarse rate function $\BARIT c(\eta,\eta')$ is appropriately defined, where $\eta\in \bar{\Sigma}$ denotes the coarse 
space variables defined by a projection operator $\COP:\Sigma \to \bar{\Sigma}$,  $\COP\sigma=\eta$.

For the sake of simplicity we present the two-level method  (ML-KMC) while every step in the study 
that follows can be easily adopted to the multilevel case.
The construction of the ML-KMC, sketched in Figure~\ref{scheme}, consists of two steps. 
In the {\it first step} we construct an approximating process on
the coarse space $ \BARIT\Sigma$ described by a generator $\BARIT \Ll$ with
rates $\BARIT{c}(\eta,\eta')$, extracting macroscopic information from the rates
$c(\sigma,\sigma') $. 
In the {\it second step} we construct rates
$\crf(\sigma'|\eta',\sigma)$, simple to simulate, and such that
$\BARIT{c}(\eta,\eta') \crf(\sigma'|\eta',\sigma)$ approximates 
$c(\sigma,\sigma')$ with an error quantified in Section~\ref{sectionError}.
Following the above description we define  a stochastic process   with rates 
 \begin{equation}\label{MRates}
  \HC(\sigma,\sigma')= \BARIT{c}(\eta,\eta') \crf(\sigma'|\eta',\sigma)\COMMA
 \end{equation}
that generates a corresponding continuous time Markov chain with transition probabilities 
  \begin{equation}\label{MultiPr}
      \widetilde p(\sigma,\sigma') =\frac{\HC(\sigma,\sigma')}{\HL(\sigma)} \COMMA \;\;\;
      \HL(\sigma) = \sum_{\sigma'\in \Sigma} \HC(\sigma,\sigma') \PERIOD 
\end{equation}
The corresponding continuous time process   $ (\{\HS_t\}_{t\ge 0},\HLl)$ is defined  
by the generator $\HLl$ on $\Sigma$, 
 \begin{equation}\label{MGen}
	\HLl \phi (\sigma) = \sum_{\sigma'\in\Sigma}
	\HC(\sigma,\sigma')\left(\phi(\sigma')-\phi(\sigma) \right)\COMMA
 \end{equation}
for every $ \phi\in L^{\infty}(\Sigma)$.
In the application we study in Section~\ref{ArrheniusMKMC} such coarse
 and reconstructing rates  are explicitly defined, 
both for exact and for controlled error sampling approaches.
In Section~\ref{sectionError} we provide  estimates that
quantify the approximating errors, for  finite   and long-time regimes, with respect to the level of 
resolution and the interpretation of the rate function decomposition.
Furthermore, controlled approximations such as \VIZ{MRates} can have
significant computational advantages, see Section~\ref{Complexity}.

\subsection{Partially rejection-free implementation of the ML-KMC method}\label{implementations}
Rejection-free methods are based on calculating (updating) rates $
c(\sigma,\sigma')$ for all $ \sigma'\in \Sigma$ at each Monte Carlo step and
choosing (searching) an event that evolves the system's state based on the
probabilities $p(\sigma,\sigma')$, see \VIZ{MicroPr}. In these methods every step
provides a successful event but the cost of implementation becomes formidable
for large complex systems where the cost increases with the system size
$|\Sigma|$ and with the complexity of the transition rates. 
Uniformization  provides a solution to this problem, suggesting a
method that needs the calculation only of one transition probability at each
step with the disadvantage of introducing rejections; a number of proposed
events will not happen. Such methods are known as null-event methods, that
generate a Markov jump process with the embedded Markov chain described by
\begin{equation}\label{NullMP}
   p^{\NULL}(\sigma,\sigma') = \begin{cases}
           1-\frac{\lambda(\sigma)}{\lambda^* }\COMMA & \mbox{if $ \sigma'=\sigma$}\\
           \frac{\lambda(\sigma) }{\lambda^*}p(\sigma,\sigma')\COMMA & \mbox{if $\sigma'\neq\sigma$}
        \end{cases}
\end{equation}
where $\lambda^*$ is a uniform upper bound of $\lambda(\sigma)$ in \VIZ{MicroPr}.
The time that the system stays at state $\sigma$, $\tau^{\NULL} $,  is governed by an exponential distribution
with  the parameter $\lambda^*\ge \lambda(\sigma)$  for all   $\sigma\in \Sigma$. 
As a result since $1/\lambda^*=\E[\tau^{{\NULL}}]\le \E[\tau_{\sigma}]=1/\lambda(\sigma)$ for all   $\sigma\in \Sigma$, 
the null event algorithm evolves the system with a smaller time step and
 needs more MC steps than a rejection free method. Despite this
inefficiency,   the significant reduction of the computational cost per MC step can be
 advantageous  for high dimensional complex systems when compared to rejection free methods.
The probability of rejecting a proposed state when the system is in the  state $\sigma$ is
 \begin{equation}\label{nullRej}
  p_{\rej}^{\NULL}(\sigma) = 1 - \frac{\lambda(\sigma)}{ \lambda^*}\COMMA
\end{equation}
and it is controlled by   $\lambda^*$, indicating that the tighter the 
upper bound is the less rejections are introduced.  
The combination of rejection-free and null-event techniques at the two  levels of the ML-KMC method 
introduces a variety of sampling techniques.   Here we propose, as the more  efficient approach,  the  rejection-free method
at the coarse level and the null-event algorithm at the microscopic level, 
balancing between an improved rejection rate of the null-event  and the computational complexity 
of the rejection-free method.   
Given $\sigma \in \Sigma$, $\eta = \COP \sigma$ the evolution of the system to 
a state $\sigma'\in \Sigma$ is achieved by the following: 

\medskip\noindent
\begin{method}\label{kMC-Null} ML-KMC   
{\rm
\begin{description}
\item[Coarse level.]
Evolve to a new  coarse state $\eta'\in \BARIT \Sigma $ with the probability
\begin{align*} 
     \BARIT p (\eta,\eta') = \frac{ \BARIT c(\eta,\eta')}{\BARIT\lambda(\eta)}\COMMA \ \ \
     \BARIT\lambda(\eta)=\sum_{\eta' \in \bar \Sigma } \BARIT c(\eta,\eta') \PERIOD
\end{align*}
\item[ Microscopic level.]
Select randomly (uniformly) $\sigma'\in \Sigma $ under the constraint $\COP \sigma' = \eta'$,
and accept it with the probability
\begin{equation*} 
     p_{\trf}(\sigma'|\eta',\sigma)=
     \frac{c_{\trf}(\sigma'|\eta',\sigma)}{\lambda_{{\trf}}(\sigma)}\COMMA\quad
     \lambda_{\trf}(\sigma)=\max_{\eta'} \sum_{\Csp} \crf(\sigma'|\eta',\sigma)\COMMA
\end{equation*}
or reject it with the probability
\begin{equation*} 
      1- \sum_{\Csp} p_{\text{rf}}(\sigma'|\eta',\sigma) \PERIOD
\end{equation*}
For implementation specifics of this step we refer to  
Section~\ref{Complexity}.
\item[Time update.]
Update time by a random time step with an exponential law with parameter 
\begin{equation}\label{UpLammda} \HL^*(\sigma) =
   \BARIT{\lambda}(\eta)\lambda_{\trf}(\sigma)\PERIOD
\end{equation}
\end{description}}
\end{method}

With  the following lemma we prove that the ML-KMC method provides correctly a
{\it partial uniformization} of the rejection-free method.
\begin{lemma}\label{MkMC_lem}
   For any    $\sigma\in \Sigma$ we have
\begin{equation*} 
   \HL(\sigma) \le  \HL^*(\sigma)  \PERIOD
\end{equation*}
\end{lemma}
\begin{proof}
{\rm 
The statement follows directly from the simple calculation
$$
   \HL(\sigma)=\sum_{\sigma'\in \Sigma} \HC(\sigma, \sigma')=
   \sum_{\eta'\in \bar\Sigma}\BARIT c(\eta, \eta') \sum_{\Csp } 
   \crf(\sigma'|\eta',\sigma) \le \BARIT\lambda(\eta)
   \lambda_{\trf}(\sigma)\PERIOD
$$
}
\end{proof}

The real time updates, controlled by
$\HL^*(\sigma) $, depend on the state of the system, while in \VIZ{NullMP}
the time step is uniform for all states, controlled by $\lambda^*$.
The multilevel method  has the rejection probability at the state $\sigma$,
\begin{equation}\label{MLRejPr}
    p_{\rej}^{\multi}(\sigma) = 1 - \frac{\HL(\sigma)}{\BARIT\lambda(\eta)
   \lambda_{\trf}(\sigma)}=1 - \frac{\HL(\sigma)}{\HL^*(\sigma)}   
   \COMMA
\end{equation}
since
\begin{align*}
     p_{\rej}^{\multi}(\sigma)& = 1 - \sum_{\sigma'\in \Sigma}\mathrm{Prob}(\sigma\to \sigma')= 1
     - \sum_{\eta'\in\bar\Sigma}\sum_{\Csp}
     \frac{\bar{c}(\eta,\eta')\crf(\sigma'|\eta',\sigma)}{\BARIT\lambda(\eta)
          \lambda_{\trf}(\sigma)} \\
      &=1 - \sum_{\sigma'\in\Sigma}
     \frac{\HC(\sigma,\sigma')}{\BARIT\lambda(\eta)\lambda_{\trf}(\sigma)}= 1 -
     \frac{\HL(\sigma)}{\BARIT\lambda(\eta) \lambda_{\trf}(\sigma)} \PERIOD\\
\end{align*}

Next, we show that the rejection rate of the ML-KMC  method can be controlled and depends on the approximation
of the coarse-grained rates.
Before stating the proposition we note that a process $\ast$ is defined as lumpable, (\cite{Lumpability}), with respect to the coarsening procedure $\eta=\COP\sigma$
when its rates satisfy
$$
\sum_{\Csp} \HC(\sigma,\sigma')=\BARIT c(\eta,\eta')\PERIOD
$$
This relation  implies uniform reconstruction rates  $\crf(\sigma'|\eta',\sigma)=1/|\Csp|$ for all $\sigma'\in \Csp$, hence 
$ \sum_{\Csp} \crf(\sigma'|\eta',\sigma) =1$, $ \lambda_{\trf}(\sigma) =1$,
\begin{eqnarray*} 
    \HL (\sigma)&=&\sum_{\sigma'\in \Sigma} \HC(\sigma,\sigma')
                =\sum_{\eta'\in \bar\Sigma}\sum_{\Csp } \HC(\sigma,\sigma')=
                   \sum_{\eta'\in \bar\Sigma}\BARIT c(\eta,\eta')
                =\BARIT\lambda(\eta) \COMMA
\end{eqnarray*}
and
$$  
   p_{\rej}^{\multi}(\sigma)= 1 - \frac{\HL (\sigma) }{\BARIT \lambda(\eta) \lambda_{\trf}(\sigma)}=0\PERIOD
$$
On the other hand, the approximating process $\ast$ in ML-KMC is not necessarily lumpable, nevertheless its approximation error determines the
rejection probability: 
 \begin{proposition}\label{Proposition1}
Let the  coarse rates define an approximately lumpable process, that is
  \begin{equation}\label{assumption1}
        \sum_{\Csp} \HC(\sigma,\sigma')=\BARIT c(\eta,\eta') + \BIGO(\epsilon)\COMMA
  \end{equation}
   uniformly in $\sigma, \eta=\COP\sigma,\eta' $ for some $\epsilon >0$.
   Then 
   $$ 
      p_{\rej}^{\multi}(\sigma)=\BIGO(\epsilon)\PERIOD
   $$
\end{proposition}

\begin{proof}
First, we denote by $\BARIT \Sigma_{\eta}=\{ \eta': c(\eta, \eta')>0\}$, that is all coarse configurations accessible in a single step from $\eta$.
Assumption \VIZ{assumption1}  implies  that
        \begin{equation}
       \BARIT c(\eta,\eta') \sum_{\Csp} \crf(\sigma'|\eta',\sigma)=\BARIT c(\eta,\eta')  + \BIGO(\epsilon)\COMMA  \sum_{\Csp} \crf(\sigma'|\eta',\sigma) = 1+\BIGO(\epsilon)\COMMA
     \end{equation}
   hence $\lambda_{\trf}(\sigma) =\max_{\eta'} \sum_{\Csp} \crf(\sigma'|\eta',\sigma)= 1+\BIGO(\epsilon)$. 
  Then
  \begin{eqnarray*}
     \HL (\sigma) = \sum_{\sigma'} \HC(\sigma,\sigma') &=&
     \sum_{\eta'} \BARIT c(\eta,\eta')\sum_{\Csp} \crf(\sigma'|\eta',\sigma) 
     = \sum_{\eta'}( \BARIT c(\eta,\eta') +   \BIGO(\epsilon) ) \\
     &=& \BARIT \lambda(\eta) + |\BARIT \Sigma_{\eta}| \BIGO(\epsilon)\COMMA
  \end{eqnarray*}
  and 
  \begin{eqnarray*}
       \frac{\HL (\sigma)}{\BARIT \lambda(\eta) \lambda_{\trf}(\sigma)}= \frac{\BARIT\lambda(\eta) + 
            |\BARIT \Sigma_{\eta}| \BIGO(\epsilon)}{\BARIT \lambda(\eta)\lambda_{\trf}(\sigma)}= 1+ \BIGO(\epsilon)\COMMA
  \end{eqnarray*}
  where in the last equality we have used   that     $\lambda_{\trf}(\sigma)=1+\BIGO(\epsilon)$ and $ \BARIT \lambda(\eta)\sim |\BARIT\Sigma_{\eta}| $ 
  in view of the definition $ \BARIT \lambda(\eta)=\sum_{\eta' \in \bar \Sigma } \BARIT c(\eta,\eta') $.
  Therefore  the rejection probability \VIZ{MLRejPr} satisfies
  \begin{equation*}
     p_{\rej}^{\multi}(\sigma) = 1 - \frac{\HL (\sigma) }{\BARIT \lambda(\eta) \lambda_{\trf}(\sigma)}
         =  \BIGO(\epsilon)\PERIOD 
  \end{equation*}
\end{proof}

%
%

\medskip\noindent
\begin{remark}
{\rm
Other implementation strategies can be designed, for instance employing   a rejection-free
method at the second level and/or a null-event method at the first level. However, when
$\lambda_{\trf}(\sigma,\eta') := \sum_{\Csp} \crf(\sigma'|\eta',\sigma) $ the
method can be implemented by a rejection-free algorithm at the microscopic level  but  
the process generated will violate the Markovian property unless 
$\lambda_{\trf}(\sigma)=\lambda_{\trf}(\sigma,\eta') $ for all $\eta'\in \BARIT{\Sigma}$.%
Another   possible modification in the algorithm  is performing a null-event  at the
coarse level combined with a null-event algorithm at the microscopic level. In this approach the need of a uniform normalization
constant due to the null event nature in both steps is a disadvantage since then the time
update is small. In fact the average time update will be inversely proportional to
$$
    \widetilde \lambda^*=N \max_{\eta,\eta'}\{\BARIT{c}(\eta,\eta') \}
    \max_{\sigma',\sigma}\{\crf(\sigma'|\eta',\sigma)\} \le \lambda^*\COMMA 
$$ 
which  indicates that the method will have higher rejection rate 
than a conventional null event algorithm.
Even though the last two implementation techniques  seem to have disadvantages, we expect that application  in case studies 
could be effective, for example when $\widetilde \lambda^*\simeq \lambda^*$  
and/or the computational acceleration due to coarsening is significant.
}
\end{remark}

\section{Applications to complex interacting particle  systems}\label{stlattice}
We demonstrate the use of the ML-KMC method in the efficient and accurate simulation  of complex systems where coarse-graining either fails or it is computationally very expensive.
In this paper we focus on the important example of systems with  competing short- and long- ranged interaction; such systems typically have  complicated phase diagrams, may exhibit pattern formation 
and arise in numerous  physical and chemical systems, \cite{Nature2001,ChVl,LEFK}. We discuss such a system next and use it as a demonstration  example for the ML-KMC method in the following Sections.

\subsection{Microscopic dynamics}\label{microdynamics}
We consider an Ising-type system on a periodic $d$-dimensional lattice $\LATT$ with $N=n^d$ lattice points. 
At each $x \in \LATT$ we can define an order parameter $\sigma(x)$. For
example, when taking values $0$ and $1$, it can describe vacant
and occupied sites. The microscopic dynamics are described by a continuous time
Markov chain with state space $\Sigma_N=\{ 0,1\}^{\LATT}$. For a configuration
$\sigma$ we denote by $\sigma^x$ the configuration which differs from $\sigma$
by an order parameter flip at the site $x$. The configuration update
$\sigma\to\sigma^x$ occurs with a rate $c(x,\sigma )$, i.e., the
order parameter at $x$ changes over the time interval $[t, t+\Delta t]$ with the
 probability $c (x,\sigma)\Delta t +o(\Delta t)$. The resulting stochastic
process $(\st,\Ll)$ is a continuous time Markov jump process with the  
generator %
defined in terms of the rate $c(x,\sigma )$ by \VIZ{Gen1}.
In this work
we present as an example the dynamics with the Arrhenius rate for the spin-flip (adsorption-desorption) 
mechanism
\begin{equation}  \label{arrh}
  c(x, \sigma)=d_0\big(1-\sigma(x)\big)+
  d_0\sigma(x)\exp\big[-\beta U(x, \sigma)\big]\COMMA
\end{equation}
where  $U(x,\sigma)= H_N(\sigma) - H_N(\sigma^x)$ is the interaction energy of a particle located at the lattice site 
$x\in\LATT$. Arrhenius laws are typically used  in micro-kinetic modeling of physiochemical applications, 
see for instance \cite{graphKMC}. The Hamiltonian function $H_N:\Sigma_N\to\R$ defines the total energy 
of the configuration $\sigma\in\Sigma_N$
and we consider pair-wise long- and short-range interactions %
\begin{equation}\label{HamSplit}
    H_N(\sigma) =\HAMS(\sigma) + \HAML(\sigma) + \sum_{x\in\LATT} h(x)\sigma(x) \COMMA
\end{equation}
where $h=h(x)$ is an external field and 
$$
\HAML = \sum_{x\neq y} J(x-y) \sigma(x)\sigma(y)\COMMA\;\;\;\;
\HAMS = \sum_{x\neq y} K(x-y) \sigma(x)\sigma(y)\PERIOD
$$
We consider interaction potentials $J$, $K:\R^d\to\R$ depending on the distance of $x$ and $y$, 
where the distance $|x-y|$ does not have to be necessarily 
measured in the Euclidean norm. Furthermore,  we assume:
\begin{eqnarray}
   \AONE\;\;    && J(x-y) \equiv \frac{1}{L^d} \VL\left(\frac{n}{L} |x-y|\right)\COMMA\;\; \mbox{$x,y\in\LATT$ and $L>0$} \label{microJL} \\
                && \VL\in C^1(\R)\COMMA\; \VL(-r) = \VL(r)\COMMA\;\VL(r)  = 0\COMMA\;\mbox{for $r>1$} \nonumber \\
   \ATWO\;\;    && \mbox{$K\in L^1_{\mathrm{loc}}(\R)$ and $K(x-y) \neq 0$ for $|x-y|\leq S$} \label{microK} \\
                && \mbox{and $S$ is independent of the lattice size $N$.} \nonumber
\end{eqnarray}
The parameters $L$ and $S$ then define a range of interactions, i.e., the number of particles interacting with a given particle at the site
$x$. The parameter $L$ can be equal to $n$, i.e., interactions with all particles. The scaling in  $\AONE$ ensures that in the 
infinite volume limit $N\to\infty$,  $J\in L^1(\R)$. %
Note that the regularity condition imposed on $\VL$ rules out singular potentials such as
Coulomb, Lenard-Jones etc. However, such cases can be treated in our analysis by splitting the potential into a smooth part (with $L=n$)
and the short-range singular part. The range parameter $S \ll n$ is fixed and independent of $n$. With this form of the Hamiltonian
the energy difference can be expressed as
\begin{equation}\label{microSL}
  U(x, \sigma)=\sum_{y\in\mathcal{I}^{(s)}(x)} K(x-y)\sigma(y) + \sum_{y\in \in\mathcal{I}^{(l)}} J(x-y)\sigma(y)-h(x) = 
               \US(x, \sigma) + \UL(x, \sigma)\COMMA
\end{equation}
where the size of the support $|\mathcal{I}^{(s)}| = S^d = \BIGO_N(1)$ and with $L\sim n$ we have $|\mathcal{I}^{(l)}| = L^d = \BIGO_N(N)$.  

\subsection{Coarse grained dynamics}\label{cgDynamics}
Systems with smooth long-range interactions are well
approximated by coarse-graining techniques, \cite{KMV1, KPS, KT}, and
CGMC are reliable and highly efficient simulation methods with controlled error
approximations. Furthermore, models where only short-range
interactions appear are inexpensive to simulate with conventional
methods as there exist algorithms with the complexity $\BIGO_N(1)$ per time step, \cite{binder}.
However, when both short and long-range interactions are
present, the conventional methods  become  prohibitively expensive,
while  coarse-graining methods   are either not easily applicable or very expensive to implement due to the necessity 
to incorporate in them  multi-body interactions, \cite{KPRT3, AKPR}.

In the series of papers \cite{KMV1, KMV, KPS} the authors initiated the
development of mathematical strategies for the {\em coarse-graining} (CG) in
stochastic lattice dynamics. One constructs the coarse lattice $\LATTC$ by dividing
 $\LATT$ in $M$ coarse cells, each of
which contains $Q=q^d$ (micro-)cells. Each coarse cell is denoted by $C_k$, 
$k \in\BARIT \Lambda_M$. A typical choice for the coarse variable in the context of
Ising-type models is the block-spin over each coarse cell $C_k$,
\begin{equation*}
   \eta:=\left\{ \eta(k)=\sum_{x \in C_k} \sigma(x) \,:\,k \in \LATTC\right\} \COMMA
\end{equation*}
defining the coarse graining map $\COP: \Sigma_N \to \CMspace$,
$\COP\sigma=\eta$ and  the coarse state space $\CMspace=\{0,1,\dots,Q\}^{\LATTC} $.   
The coarse grained approximating process $\etat$, \cite{KMV1, KMV, KV}, is defined by
adsorption and desorption rates of a single particle in the coarse cell $C_k$
\begin{equation}\label{cgdes}
  \BARIT{c}_a(k, \eta)=d_0(Q-\eta(k))\COMMA
   \quad \BARIT{c}_d(k, \eta)=d_0\eta(k)\exp\big[-\beta\bar U(k,\eta)\big]\COMMA
\end{equation}
where the CG interaction potential is given by 
\begin{equation*}
     \BARIT U(k,\eta)=\sum_{l \ne k\,,\, l\in \LATTC} \bar J(k, l) \eta(l)+
               \BARIT J(k, k) (\eta(k)-1)-\bar h(k)\,,
\end{equation*}
where for the coarse cells $k$, $l\in\LATTC$ we define
\begin{equation}\label{Jbar}
   \bar{J}(k,l)=\frac{1}{Q^2}\sum_{x\in C_k}\sum_{y\in C_l} J(x-y)\COMMA\;\;
   \bar{J}(k,k)=\frac{1}{Q(Q-1)}\sum_{x\in C_k} \sum_{y\in C_k,y\neq x}J(x-y)\COMMA
\end{equation}
as the interaction potential on the coarse space.

\subsection{Long-time behavior and the stationary measure}
In many applications of  interacting particle systems the long-time behavior of (ergodic) evolution is
characterized by the stationary (equilibrium) measure. 
If the rates $c(x,\sigma)$ satisfy the reversibility condition with respect to the measure $\mu_{N.\beta} = Z_N^{-1} \EXP{-\beta H_N(\sigma)}$, 
also known as {\em  detailed balance}, 
\begin{equation}\label{DB}
    c(x,\sigma) \EXP{-\beta H_N(\sigma)} = 
    c(x,\sigma^x) \EXP{-\beta H_N(\sigma^x)}\COMMA
\end{equation}
then the jump dynamics leaves the Gibbs measure
\begin{equation}\label{gibbs}
   \mu_{N, \beta}(d \sigma) = \frac{1}{Z_{N}}\EXP{-\beta H_N(\sigma)}\,P_N(d\sigma)
\end{equation}
invariant, as is the case for \VIZ{arrh}. The factor $Z_{N}$ is the normalizing constant (partition function). 
Furthermore, the product Bernoulli distribution $P_N(d\sigma)$, is the
prior distribution on $\LATT$ representing distribution of
states in a non-interacting system.  
The total energy $H_N(\sigma)$ of the system, at the configuration $\sigma=\{ \sigma(x):x\in\LATT\}$,
is given by the Hamiltonian $H_N$ \VIZ{HamSplit}.

The 
coarse-grained  process \VIZ{cgdes} satisfies detailed balance ensuring that
the process, at least for long-range potentials $J$,  has as its invariant measure an approximation of the coarse-grained Gibbs measure,
\cite{KPRT},  
\begin{equation}\label{muzero}
   \BARMZ_{M,\beta}(d\eta) = \frac{1}{\BARIT{Z}_M^{(0)}} \EXP{ -\beta \BARHZ(\eta)}\BARIT{P}_M( d{\eta})\COMMA
\end{equation}
where  %
\begin{equation}\label{Hzero}
  \begin{split}
\BARHZ(\eta)&=-\frac{1}{2} \sum_{l \in \LATTC}\sum_{\substack{k \in \LATTC\\ k\ne l}}\bar J(k,l)\eta(k)\eta(l)
             -\frac{1}{2} \bar J(0, 0) \sum_{l \in \LATTC}\eta(l)\left(\eta(l)-1\right) \\
            &+\sum_{k \in \LATTC}\BARIT h(k)\eta(k)\PERIOD
  \end{split}
\end{equation}
Applying the same coarse-graining formula \VIZ{Jbar} to the short range potential $K$ will introduce errors that
are not well-controlled as $N\to\infty$, \cite{KPRT3}. However, the multi-level technique provides an approach for
constructing approximations that do not require higher-order (cluster) expansions of the short-range potentials
developed in \cite{KPRT3}. For use of the multi-level approach in the context of equilibrium sampling and corresponding
error analysis we refer the reader to \cite{KKP, KKPV}. We also revisit this error analysis in Section~\ref{sectionError}, noting that we do not require 
the reversibility condition \VIZ{DB} for the simulated process in our results.

\subsection{ML-KMC method  for Arrhenius dynamics}\label{ArrheniusMKMC}
As a specific example we explain the ML-KMC method for sampling the microscopic process  $\st$ generated  
by the Arrhenius rate for the adsorption-desorption mechanism.
For the model  considered the coarse space rate functions are explicitly
given by the coarse graining technique in Section~\ref{cgDynamics}. 
The   reconstruction rates that we present rely on 
 approximate sampling with
potential splitting where the long and short-range interactions are split to the
first and second level respectively and no corrections to coarse-grained rates are applied. 

\noindent 
{\it Approximate dynamics with potential splitting.}\label{ApprSplit}
In this variant only the sampling  
corresponding to the costly long-range interactions is performed at the first (coarse) level
and the compression of the short-range interaction is avoided. The sampling of the short-range
contributions is performed at the second (fine) level. 
The rates on the coarse space $\CMspace$  
\begin{equation*}
   \bar{c}_a(k,\eta) = d_0\left(Q-\eta(k)\right)\COMMA\;\;\; \bar{c}_d(k,\eta) = d_0
   \eta(k) \EXP{-\beta \BARUL(k,\eta)}\COMMA
\end{equation*}
where $\BARUL(k,\eta) = \sum_{\substack{l\in \LATTC \\\ l\neq k}}\BARIT
J(k,l) \eta(l) + \BARIT J(k,k)(\eta(k)-1) -\frac12\BARIT{h}(k)$. Since $\sigma'=\sigma^x$ is a spin flip updating, the
reconstruction rates  in \VIZ{MRates} are explicitly defined and denoted by
\begin{equation}\label{ApproxArrh}
   \crf^a(x|k,\eta)=\frac{1-\sigma(x)}{Q-\eta(k)}\COMMA\;\;\;
   \crf^d(x|k,\eta)=\frac{\sigma(x)}{\eta(k)}\EXP{-\beta \US(x,\sigma)}\COMMA
\end{equation}
where
\begin{equation*} 
    \US(x,\sigma) = \sum_{y \ne x, y \in \LATT} K(x-y)\sigma(y)-\frac12 h(x)\COMMA \;
    \UL(x,\sigma) = \sum_{y \ne x, y \in \LATT} J(x-y)\sigma(y)-\frac12 h(x)\PERIOD
\end{equation*}
Note that $\crf^a(x|k,\eta) $ (and $\crf^d(x|k,\eta)$) are well-defined since $ \eta(k)\neq Q$ 
($\eta(k)\neq 0$) for all $k\in \LATTC, x\in C_k$ when
adsorption (desorption) process is selected in the cell $k$.
With this choice of rates the ML-KMC method generates a Markov process $(\ast,\HLl)$ with the rate function defined by
\begin{eqnarray}
\HC(x,\sigma)& =& \bar{c}_a(k,\eta) \crf^a(x|k,\eta) + \bar{c}_d(k,\eta)   \crf^d(x|k,\eta) \nonumber \\
  &=&d_0 (1-\sigma(x)) + d_0  \sigma(x) \EXP{-\beta  \HU(x,\sigma)} \label{ApprRates}\COMMA
\end{eqnarray}
where we define
\begin{equation}\label{Uhat}
     \HU(x,\sigma)= \US(x,\sigma) + \BARUL(k,\eta)\COMMA\;\; x\in C_k\COMMA \eta=\COP\sigma\PERIOD
\end{equation}
In Appendix~\ref{AppendixDB} we prove that $\HC(x,\sigma)$ satisfy the detailed balance condition with 
\begin{equation}\label{approxGibbs}
    \widetilde\mu_{N, \beta}(d \sigma)
        = \frac{1}{\widetilde{Z}_{N}}\EXP{-\beta \widetilde  H_N(\sigma)}\,P_N(d\sigma)\COMMA
 \end{equation}
and $\widetilde{Z}_{N}$ is the normalization constant corresponding to the Hamiltonian
\begin{eqnarray}\label{Hhat}
     &\widetilde H_N(\sigma) =& -\frac12 \sum_{x \in \LATT}\sum_{y\not= x}
          K(x-y)\sigma(x)\sigma(y) -\frac12\sum_{x \in \LATT}\sum_{y\not= x} \BARIT
          J(k(x),l(y))\sigma(x)\sigma(y) \nonumber \\ & &+\sum_{x\in\LATT} h(x)\sigma(x)\PERIOD
\end{eqnarray}
We define $k(x)$ to be the coarse cell $k\in\LATTC$ such that $x\in C_k$.
\section{ Controlled-error approximations for complex systems}\label{sectionError}
In this section we provide estimates that quantify the numerical error when approximating
the continuous time Markov process
$(\st, \Ll) $ by $ ( \ast, \HLl )$ defined by \VIZ{ApprRates}, with invariant
stationary measures $\mu_{N,\beta}(d\sigma)$, \VIZ{gibbs}, and $\widetilde \mu_{N,\beta}(d\sigma)$, 
\VIZ{approxGibbs}, respectively.
We prove information loss estimates in long-time stationary regimes and weak
error estimates for suitably defined {\em macroscopic} observables in finite time intervals. 
Finally, as it is evident from the proofs, our results for stationary dynamics (processes)  and weak error estimates,
 hold true for the ML-KMC approximation of 
general particle systems which are not necessarily 
reversible.

\subsection{Controlled approximations at long times} \label{errors2}
We analyze approximation properties in long-time,
stationary regimes, including (a) the behavior of the stationary measure and its
sampling, as well as  (b) the stationary process, i.e., dynamics where the initial
measure is the stationary distribution reached after long-time integration. The stationary
dynamics present an especially important regime describing dynamic transitions between
metastable states in complex energy landscapes, while at this regime we construct the system's phase diagrams, 
see the simulation of hysteresis in  Figure~\ref{fig:2}.

The  error analysis in the stationary regime is the primary novelty in our paper from a numerical analysis perspective
and it is 
not restricted only to the ML-KMC algorithms; these questions  are ubiquitous for
numerical approximations of reversible and irreversible stochastic dynamics such as Langevin stochastic
differential equations, thermostated dynamics, dissipative particle dynamics  methods, etc. 
We expect that  our proposed entropy-based
perspective  could be used to assess such numerical schemes at the stationary
dynamics regime in a quantifiable manner, in a variety of stochastic,  extended as well as finite-dimensional 
systems.

\smallskip
\noindent
(a) {\it Estimates for the stationary measure.} For reversible systems, the
explicit knowledge of the invariant measures of the ML-KMC process  allows us to
compare them directly to the Gibbs states associated with reversible kMC.
Error estimates are given  in terms
of the specific relative entropy 
$$\Rr(\mu| \nu)\equiv N^{-1}\int_\Sigma\log\big\{d\mu(\sigma)/ d\nu(\sigma) \big\} \mu(d\sigma)$$ 
between the corresponding equilibrium Gibbs measures. 
The scaling factor $N^{-1}$ is related to the extensivity of the system, hence the
proper error quantity that needs to be tracked is the  loss of
information {\it per particle} $\Rr(\mu_{N,\beta} |\widetilde\mu_{N,\beta})$.
Relative entropy  was   used as
measure of loss of information in coarse-graining in \cite{KPRT, KPRT2},
and as means for sensitivity analysis in the context of climate modeling
problems, \cite{Majda}.
One of the results in \cite{KPRT,KT} concerns derivation of the loss of
information  per particle estimates  on the coarse space
for smooth long-range interactions $J$, \VIZ{microJL}.
In the following theorem we prove an analogous error estimate on the microscopic space taking into account both
 short and long-range interactions.

\begin{theorem}\label{StatEstimate}
Let $\widetilde\mu_{N,\beta}$ be the approximating measure of the microscopic equilibrium measure $\mu_{N,\beta} $ 
defined by~\VIZ{approxGibbs} and~\VIZ{gibbs}, with Hamiltonian functions~\VIZ{Hhat} and \VIZ{HamSplit} respectively, 
then the loss of  information per particle is estimated by
\begin{equation*}
   \Rr(\mu_{N,\beta} |\widetilde\mu_{N,\beta})=\BIGO_N \left(\beta \frac{q}{L} \| \nabla \VL \|_\infty\right) \PERIOD
\end{equation*}
\end{theorem}

Before continuing with the proof of the theorem we state a necessary estimate in   
Lemma~\ref{Jestimate}  that is  proved in \cite{KPRT}.  
\begin{lemma}\label{Jestimate}
   Assume the interaction potential $J$ satisfies $\AONE$ in  \VIZ{microJL}, then the coarse-grained interaction
   potential $\BARIT J$, given by  \VIZ{Jbar} at the coarsening level $q$ approximates for any $x$, $y\in\LATT$ and
   $k$, $l\in\LATTC$ the potential $J$ with the error
   \begin{eqnarray*}
       &&| J(x-y)-\BARIT J (k,l)|\le 2\frac{q}{L}\sup_{x'\in C_k, y'\in C_l\atop y'\neq x'} \| \nabla \VL(x'-y')\| 
                 \le C_V \frac{q}{L^2}\COMMA
\end{eqnarray*}
where the constant $C_V$ is independent of $q$, $L$.
\end{lemma}

\begin{proof}[Theorem~\ref{StatEstimate}]
\begin{eqnarray*}
  \Rr(\widetilde\mu_{N,\beta} |\mu_{N,\beta})&=&\frac1N \int_{\Sigma_N} \log\left( \frac{d\widetilde\mu_{N,\beta}}{d\mu_{N,\beta}}\right)\widetilde\mu_{N,\beta}(d\sigma) \\
&=&\frac1N\log\frac{ Z_N}{\widetilde Z_N} + \frac1N\E_{\widetilde\mu_{N,\beta}}\left[ \beta( H_N(\sigma)-  \widetilde H_N(\sigma))\right]\PERIOD
\end{eqnarray*}
By the definition of the    Hamiltonian and Lemma~\ref{Jestimate}  we have the estimate
\begin{equation*}
\frac1N| H_N(\sigma)-\widetilde H_N(\sigma)|= \frac{1}{N}| \HAML(\sigma)-\WHAML(\eta)| \le C \frac{q}{L} \| \nabla \VL\|_\infty\COMMA
\end{equation*}
where $ \eta=\COP\sigma$ and $C$ is a positive constant independent of the system size $N$, the coarsening parameter $q$, and the short range potential $K=K(x-y) $.
Therefore
\begin{eqnarray*}
 \frac1N\log\frac{ Z_N}{\widetilde Z_N} =\frac1N\log\E_{\widetilde\mu_{N,\beta}}\left[ \exp\{-\beta(H_N(\sigma)-\widetilde H_N(\sigma))\}\right]=
              \BIGO_N \left(\beta \frac{q}{L} \| \nabla \VL \|_\infty\right)\COMMA\\
\text{and} \quad \frac1N\E_{\widetilde\mu_{N,\beta}}\left[ \beta( H_N(\sigma)-  \widetilde H_N(\sigma))\right]=
              \BIGO_N \left(\beta \frac{q}{L} \| \nabla \VL \|_\infty\right)\COMMA
\end{eqnarray*}
that concludes the proof.
\end{proof}

Theorem~\ref{StatEstimate} proves that the potential splitting does not affect  equilibrium properties of  the system, 
a fact that we indeed observe in the numerical experiment, see for example Figure~\ref{fig:8}.  The error estimate is independent of the short-range 
interaction potential as was expected, since  the approximation of the invariant measure results only from compressing  the long-range interactions. 

\medskip
\noindent
(b) {\it Approximating the stationary process.}
The analysis stems from attempting to understand the striking accuracy of
ML-KMC in calculating phase diagrams, hysteresis simulations, Figure~\ref{fig:2} and     Figure~\ref{fig:9},  as well as in dynamic transitions
between metastable states, see Figure~\ref{fig:1} and     Figure~\ref{fig:6}.
We assess the approximation of the kMC process $\StT$ by the ML-KMC $\aStT$, when the initial data are sampled from a stationary measure.
We consider the relative entropy per particle  formula %
 in the time interval $[0, T]$
$$
   \RELENT{{\Dd}_{[0,T]}}{{\tilde \Dd}_{[0,T]}}
      =\frac{1}{N}\int \log\left(\frac{d{\Dd}_{[0,T]}}{d{\tilde\Dd}_{[0,T]}}\right)
         d{\Dd}_{[0,T]}\COMMA
$$
where ${\cal D}_{[0,T]}$ 
(resp. ${\cal \tilde D}_{[0,T]}$) is the distribution of a process $\StT$
(resp. $\aStT$) on the  path  space $\mathcal{Q}([0,T],\Sigma_N)$, the space of right continuous with left limits $\Sigma_N$-valued functions
defined on $[0,T]$. We note that the relative entropy measures the loss of information in the approximation
of the kMC process $\StT$ by the ML-KMC $\aStT$. 

If the initial distribution of the   process $\StT$ (resp.  $\aStT$) is  the stationary measure
 $ \mu$ (resp. $\tilde \mu$), %
 then the specific relative entropy simplifies to the following relation, \cite{CGREJP2006},
\begin{equation}\label{relent1}
   \RELENT{{\Dd}_{[0,T]}}{{\tilde \Dd}_{[0,T]}}
    =T \Hh({\Dd}_{[0,T]}|{\tilde \Dd}_{[0,T]})+\RELENT{\mu}{\tilde \mu}\COMMA
\end{equation}
where $\RELENT{\mu}{\tilde\mu}$ is the specific relative entropy between the stationary
measures, and
\begin{equation}\label{relent2}
   \Hh({\Dd}_{[0,T]}|{\tilde \Dd}_{[0,T]})={1 \over N} \E_\mu \left[\lambda (\sigma) - \tilde \lambda(\sigma) -
         \sum_{\sigma'} \lambda(\sigma) p(\sigma, \sigma')\log \frac{\lambda(\sigma)
         p(\sigma, \sigma')}{\tilde \lambda(\sigma) \tilde p(\sigma, \sigma')}\right]\COMMA
\end{equation}
 is given in terms of the jump rates $\lambda$, $\tilde \lambda$ and
jump probabilities $p$, $\tilde p$ of $\StT$ and  $\aStT$ respectively. 
Indeed, by Girsanov's formula,  \cite{KL}, we obtain the corresponding Radon-Nikodym derivative
\begin{equation}\label{RN1}
   {\frac{d{\cal D}_{[0,T]}}{d{\cal \tilde D}_{[0,T]}}}(\rho_t) \!\!=\!\!
   \frac{\mu(\rho_0)}{\tilde \mu(\rho_0)}\exp \left\{ \int_{0}^{T}\! \!\!\![\lambda(\rho_s) -\widetilde \lambda (\rho_s)]ds - \!\!
   \int_0^T \!\!\!\! \sum_{\sigma\in \Sigma}\!\!p(\sigma,\rho_s)\log \frac{\lambda(\sigma)
     p(\sigma,\rho_s)}{\widetilde\lambda(\sigma)\widetilde p(\sigma,\rho_s)} dN_s(\rho) \right\}
\end{equation}
on any path $\{\rho_t\}_{t \in [0,T]}$ in $\mathcal{Q}([0,T],\Sigma_N)$, where $N_s(\rho)$ is the number of jumps of the path $\rho$ up to time $s$.
Then for any continuous, bounded function $\phi:\Sigma_N \to \R$
\begin{eqnarray*}
   \E_{\Dd}\left[\int_0^T \phi(\rho_s)\,dN_s(\rho)\right]=\E_{\Dd}\left[\int_0^T \phi(\rho_s)\lambda(\rho_s)\,ds \right]= T \E_{\mu} [ \phi \lambda]\COMMA
\end{eqnarray*}
where $\E_{\Dd}[\phi(\rho_s)]=\int \phi(\rho_s)d{\Dd}_{[0,T]}$, the first equality results from the fact that $N_t-\int_0^t\lambda(\rho_s)\,ds $ is a (zero mean) martingale and 
the second equality follows because $\{\rho_t\}_{t\in[0, T]}$ is a stationary process, i.e., $\E_{\Dd}[\phi(\rho_t)]=\E_{\mu}[\phi(\rho_t)] $.
Thus we have
\begin{multline*}
\RELENT{{\Dd}_{[0,T]}}{{\tilde \Dd}_{[0,T]}} = N^{-1}\E_{\Dd}\left[\log{ d{\cal D}_{[0,T]}\over d{\cal \tilde D}_{[0,T]}}\right] 
                                             = N^{-1} \E_{\Dd}\left[\log\frac{\mu}{\tilde\mu}\right] \\
                                             + N^{-1}\E_{\Dd}\left[\int_{0}^{T} [\lambda(\rho_s) -\widetilde \lambda (\rho_s)]\,ds-
                                                       \!\!\!\int_0^T \!\!\! \sum_{\sigma\in \Sigma}p(\sigma,\rho_s)\log \frac{\lambda(\sigma)
                                                  p(\sigma,\rho_s)}{\widetilde\lambda(\sigma)
                                                  \widetilde p(\sigma,\rho_s)} \,dN_s(\rho)\right] \\
                                             =T N^{-1} \E_\mu \left[\lambda (\sigma) -\tilde \lambda(\sigma) -
                                                 \sum_{\sigma'} \lambda(\sigma) p(\sigma, \sigma')\log \frac{\lambda(\sigma)
                                                 p(\sigma, \sigma')}{\tilde \lambda(\sigma) \tilde p(\sigma, \sigma')}\right] +\RELENT{\mu}{\tilde \mu} \\
                                             = T \Hh({\Dd}_{[0,T]}|{\tilde \Dd}_{[0,T]}) + \RELENT{\mu}{\tilde \mu} \COMMA
\end{multline*}
which is the formula \VIZ{relent1}.

\begin{remark}
{\rm
Formula~\VIZ{relent1}  shows  that in the stationary dynamics regime the information loss consists of two terms, 
one which scales as $\BIGO_T(1)$ in $T$ and is related to the stationary measures $\mu$ and $\tilde \mu$  
and another one that captures the stationary dynamics and scales as $\BIGO_T(T)$. 
Furthermore, we note that for the stationary process approximation the relevant quantity is the {\em relative entropy per unit time} 
$\Hh({\Dd}_{[0,T]}|{\tilde \Dd}_{[0,T]})$, \VIZ{relent2}. On one hand, \VIZ{relent1} implies that the loss of information increases linearly in time, 
while  the stationary measure loss of information  becomes irrelevant as $T \to \infty$. 
The fact that as $T$ grows in \VIZ{relent1} the   term $\RELENT{\mu}{\tilde \mu}$ becomes unimportant is especially useful since  $\tilde \mu$ is typically  
not known explicitly (contrary to the case in \VIZ{approxGibbs}),  while in non-reversible systems (e.g., reaction-diffusion kMC) $\mu$ is  not known either, however, 
due to \VIZ{relent1} it is not necessary to  calculate or estimate $\RELENT{\mu}{\tilde \mu}$. 
}
\end{remark}

We can use this information theory-based  perspective to evaluate broad classes of numerical schemes for
extended stochastic processes such as the ones arising in kMC, in the 
long-time, stationary process regime. However, here the following theorem provides the information loss estimates for 
approximating the microscopic process $(\StT, \Ll) $
with  $(\aStT, \HLl)$
defined by \VIZ{arrh} and \VIZ{ApprRates} respectively.
\begin{theorem}\label{unitTime}
[{\it A priori estimates}.]
 Let $(\StT, \Ll) $   and $(\aStT, \HLl)$ be stationary processes with the initial distributions $\muN$  and $ \mutN$ respectively, then
\begin{enumerate}[(a)]
\item For any fixed time $T>0$
  \begin{equation} \label{Gibrelent}
      \RELENT{{\Dd}_{[0,T]}}{{\tilde \Dd}_{[0,T]}}=T \Hh({\Dd}_{[0,T]}|{\tilde \Dd}_{[0,T]})+\RELENT{\muN}{\mutN}
  \end{equation}
  where 
  \begin{equation}\label{relentTh}
     \Hh({\Dd}_{[0,T]}|{\tilde \Dd}_{[0,T]})={\frac{1}{N}} \E_\mu \left[\lambda (\sigma) - \tilde \lambda(\sigma) -
         \sum_{\sigma'} \lambda(\sigma) p(\sigma, \sigma')\log \frac{\lambda(\sigma)
         p(\sigma, \sigma')}{\tilde \lambda(\sigma) \tilde p(\sigma, \sigma')}\right]\PERIOD
  \end{equation}
\item For any $N$, coarsening parameter $q<L$ and interaction potentials $J(x-y)$, $K(x-y)$ 
      satisfying $\AONE$, \VIZ{microJL},  and $\ATWO$, \VIZ{microK}, respectively, 
\begin{equation}\label{UnitTimeEst}
    \Hh({\Dd}_{[0,T]}|{\tilde \Dd}_{[0,T]}) \le  \beta\frac{q}{L} C(K,V,\beta)
    \|\nabla \VL\|_1 \COMMA
\end{equation} 
where   $\|\cdot\|_1\equiv\|\cdot\|_{L^1}$ is the $L^1$ norm  on $\R$ and    
$C(K,V,\beta)=C(\|K\|_\infty,\|\VL\|_\infty,\beta)$ is a constant independent of $N$.
\end{enumerate}
\end{theorem}

\begin{proof}
{\it  (a)} Relation~\VIZ{Gibrelent} is a direct consequence of the earlier discussion.

\noindent
{\it (b)} Recalling the definition of the Markov jump process for the microscopic  and the approximating process we have  
for all $x\in \LATT$, $\sigma\in \Sigma_N $ that $\lambda(\sigma)p(\sigma,\sigma^{x}) =c(x,\sigma)$, $\lambda(\sigma)=\sum_{x\in\LATT}c(x,\sigma)$ 
and $\HL(\sigma)\tilde p(x,\sigma) =\HC(x,\sigma')$, $\HL (\sigma)= \sum_{x\in\LATT}\HC(x,\sigma)$, with the rate functions 
$c(x,\sigma) =d_0(1-\sigma(x)) +d_0\sigma(x)\EXP{-\beta U(x,\sigma)}$ and 
$\HC(x,\sigma) =d_0(1-\sigma(x)) +d_0\sigma(x)\EXP{-\beta\HU(x,\sigma)}$ as defined in \VIZ{arrh} and \VIZ{ApprRates} respectively. 
Then according to formula~\VIZ{relent2}  we have
\begin{equation*}
\Hh({\Dd}_{[0,T]}|{\tilde \Dd}_{[0,T]})=N^{-1}\E_{ \muN}
\Big[ (\lambda (\sigma) - \HL(\sigma)) -
\sum_{x\in\LATT} \ c(x,\sigma)\log \frac{  c(x,\sigma)}{
\HC(x,\sigma)}\Big]\PERIOD
\end{equation*}
We define 
$\Delta_{q,N}(x,\sigma)\equiv U(x,\sigma)-\HU(x,\sigma)$.  
From the definition of $\HU(x,\sigma)$, \VIZ{Uhat}, it follows that 
$\Delta_{q,N}(x,\sigma)= \UL(x,\sigma)-\BARUL(k, \COP\sigma)$, for all $x\in C_k$.
In view of this equality a straightforward application of
Lemma~\ref{Jestimate} states that there exists a constant $c>0$ such that the microscopic potential $ U(x,\sigma)$ is approximated 
by $\HU(x,\sigma) $ with
\begin{equation}\label{DeltaU}
   |\Delta_{q,N}(x,\sigma)|\le  c  \frac{q}{L} \|\nabla \VL \|_\infty, \text{ for all } \sigma \in \Sigma_N, x\in
   \LATTN \PERIOD
\end{equation}
Then
\begin{align*}
    &\Hh(\Qm | \HQm)=   N^{-1} \E_{ \muN}\left[\sum_{x\in \LATT} e^{-\beta U(x,\sigma)}
                         \sigma(x)\left(1-e^{-\beta\Delta_{q,N}(x,\sigma)}\right)\right] \\
                    &\;\;\;\;-   N^{-1}\E_{\muN}\left[  \sum_{x\in \LATT\atop \sigma(x)=1} d_0\sigma(x)e^{-\beta U(x,\sigma)} \log\frac{ e^{-\beta U(x,\sigma)}}{  e^{-\beta \HU(x,\sigma)}}\right]\\
                    &\le N^{-1}C(K,V,\beta) %
                    \E_{ \muN}\left[\sum_{x\in \LATT} \beta\Delta_{q,N}(x,\sigma)  \right]\\
                    &\;\;\;\;+   N^{-1}\E_{\muN}\left[  \sum_{x\in \LATT\atop \sigma(x)=1} d_0\sigma(x)e^{-\beta U(x,\sigma)} \Delta_{q,N}(x,\sigma) \right] \\
                    &\le 2 c \beta \frac{q}{L}  C(K,V,\beta)
                    \|\nabla \VL \|_1   \COMMA
\end{align*}
where $C(K,V,\beta)=
\sup_{\sigma,x}\exp\{-\beta U(x,\sigma)\}$.
\end{proof}

\begin{remark}
{\it [A posteriori error analysis]}\, 
{\rm 
Reversing the roles of $\mu$ and $\tilde \mu$ in the formula \VIZ{relentTh} we obtain an {\em a posteriori} calculation on the loss of information in \VIZ{Gibrelent}:
\begin{equation}\label{relent2Ap}
    \Hh({\tilde \Dd}_{[0,T]}|{\Dd}_{[0,T]})={1 \over N} \E_{\tilde\mu}
                 \Big[(\tilde\lambda (\sigma) - \lambda(\sigma)) -
                 \sum_{\sigma'}\tilde\lambda(\sigma) \tilde p(\sigma, \sigma')\log
                 \frac{\tilde\lambda(\sigma) \tilde p(\sigma, \sigma')}{ \lambda(\sigma) p(\sigma, \sigma')}\Big]\PERIOD
\end{equation}
Indeed, viewing this function as an observable estimated on the approximating stationary process, we note that  it can be computed {\em a posteriori} in the course of an 
ML-KMC simulation by sampling  from the stationary measure $\tilde\mu$. 
We note that in \cite{KPRT} we derived and tested
computationally  a posteriori estimates for adaptive coarse-graining of
extended systems, based on a similar relative entropy approach for sampling the stationary distributions. 
The {\em a posteriori} representation in \VIZ{relent2Ap} is general and applies to both reversible and irreversible processes and does not require the {\em a priori} estimates in Theorem~\ref{unitTime} (b)(c). 
The complexity of
numerical calculation of \VIZ{relent2Ap} depends on the complexity of the studied model.
For example, in the lattice systems we study here, the loss of information per unit time is sampled  in the course of a ML-KMC simulation for $T\gg 1$ as%
\begin{equation*}
    \frac{1}{T} \RELENT{{\tilde \Dd}_{[0,T]}}{{\Dd}_{[0,T]}}\approx  \Hh({\tilde \Dd}_{[0,T]}|{\Dd}_{[0,T]})=\frac{1}{N} 
                \E_{\tilde\mu} \Big[(\tilde\lambda (\sigma) - \lambda(\sigma)) -
                          \sum_{x\in\LATT} \tilde c(x,\sigma)\log \frac{\tilde c(x,\sigma)}{c(x,\sigma)}\Big]
\end{equation*}
From the practical point of view we have to deal with the somewhat computationally costly summation over the lattice
which, in principle, has  the complexity of $\BIGO(N)$.
}
\end{remark} 

\subsection{Weak error estimates in finite time}
In this section we prove weak error estimates of the approximation for a class of {\em macroscopic} observable  quantities defined below. 
The weak error is defined by $e_w = |\Es[\phi(\sigma_t)] - \Es[\phi(\HS_t)] |$ for an  observable $ \phi $ 
on the microscopic configuration space $\Sigma_N$, 
where the expectation is defined for the path conditioned on the initial configuration $\sigma_0$.
We  provide a quantitative measure of the controllable approximation that
depends on two features: (a) the coarsening level $q$ and (b) the potential
splitting. The explicit  dependence  on the strength of the short range interactions provides us  a measure to control  splitting 
of the interactions into short and long-range parts. 

\begin{theorem}\label{weakThm}
Let $( \st, \Ll)$ be the Markov process generated by the conventional
kinetic Monte Carlo method, and $( \ast, \HLl)$ the process generated by
the ML-KMC method, both with initial condition
$\sigma_0$. For any macroscopic observable, i.e. a function  $\phi \in L^{\infty}(\Sigma_N)$,  such that when 
$\partial_x \phi(\sigma):=\phi(\sigma^x) - \phi(\sigma)\, ,$
\begin{equation}\label{macroscopic}
\sum_x \|\partial_x \phi \|_{\infty} \le C < \infty\, , \quad\mbox{ where $C$ is independent of $N$}\, ,
\end{equation} 
the weak error satisfies, for $0<T<\infty$,
\begin{equation}\label{weakError}
 |\Es[\phi(\sigma_T)] - \Es[\phi(\HS_T)] | \le  C(K,V,\beta)
 \,C_T \frac{q}{L} 
\end{equation}
where $C_T$ is a constant independent of the system size  and 
$
C(K,V,\beta)=K_S J_L$,  $K_S = |\sup_{x,\sigma} \EXP{-\beta \US(x,\sigma)}|$
and $J_L = | \sup_{x,\sigma} \EXP{-\beta \BARUL(k(x),\COP\sigma)}| $.  
\end{theorem}

Some typical macroscopic observables satisfying \VIZ{macroscopic} are the coverage,
$c(\sigma_t)
=\frac{1}{N}\sum_{x\in\LATT}\sigma_t(x)$, the spatial correlations $f(\sigma ; k) = \frac{1}{N}\sum_{x\in\LATT} \sigma(x)\sigma(x+k)$,
and the Hamiltonian  defined in \VIZ{HamSplit}.

For the proof of the theorem we will need the following Lemma~\ref{lemU} that we prove in Appendix~\ref{ProofOfLemma}, see also \cite{KPS}.
We define $ u(t,\sigma_0)= \E[\phi(\sigma_T)|\sigma_t=\sigma_0]$ and
the function $u(t, \sigma_0)$ solves the backward Kolmogorov equation, i.e., the final value problem
\begin{equation}\label{pdeU}
 \partial_t u(t,\sigma)+\Ll u(t,\sigma) =0\COMMA\;\;\; u(T,\cdot)=\phi\COMMA\;\; t<T
\end{equation}
For all  observables $\phi$ satisfying \VIZ{macroscopic} we can estimate
$\partial_x u(t,\sigma)=u(t,\sigma^x) - u(t,\sigma)$ independently of $N$ since we have the following estimate.
\begin{lemma}\label{lemU}
Let $ u(t,\sigma)$ a solution of \VIZ{pdeU}
where $\Ll$ is the infinitesimal generator $\Ll f(\sigma) = \sum_x c(x,\sigma) (f(\sigma^x) - f(\sigma))$
defined by the rate function $c(x,\sigma)$ given in
\VIZ{arrh}. Then for any $t\le T$ 
\begin{equation}
     \sum_x \|\partial_x u(t,\cdot)\|_{\infty} \le C_T \sum_x \|\partial_x \phi \|_{\infty}
\end{equation}
\end{lemma}
We continue with the proof of Theorem~\ref{weakThm}.
\begin{proof}
Using the martingale property we have for any smooth function $v(t,\sigma_0)$
and the process $\{\HS_t\}_{t\ge 0} $ with the generator $ \HLl$
\begin{equation*}
   \Es[v(T,\HS_T)] = \Es[v(0,\HS_0)] + \int_0^T \Es[(\partial_s + \HLl)v(s,\HS_s))]ds\PERIOD
\end{equation*}
Therefore 
\begin{eqnarray*}
\Es[\phi(\sigma_T)] -\Es[\phi(\HS)] & =&  \Es[u(0,\sigma_0)] -\Es[u(T,\HS_T)] \\
     &=&  \int_0^T \Es[(\partial_s + \HLl) u(s,\HS_s)]\, ds   \\
     &=&  \int_0^T \Es[ \HLl u(s,\HS_s) - \Ll u(s,\HS_s)]\, ds  \\
     &=& \int_0^T \Es\left[\sum_{x\in \LATT} \left( \HC(x,\HS_s) - c(x,\HS_s)\right)\partial_x u(s,\HS_s)\right]\, ds\PERIOD
\end{eqnarray*}
However, we can bound the last term by 
$$
 \|c -\HC\|_{\infty} \int_0^T\Es\left[\sum_{x\in \LATTN }|\partial_x u(s,\HS_s)| ds \right]
$$
We conclude by 
using Lemma~\ref{lemU} and noting that
$$
|\HC(x,\HS_s) - c(x,\HS_s)|\le d_0e^{-\beta \BARUL(k(x),\COP\HS))}\left| e^{-\beta \BARUL(x,\COP\HS_s)} -e^{-\beta \UL(x,\HS_s)} \right|\COMMA
$$
where using \VIZ{DeltaU},
$$\left| e^{-\beta \BARUL(x,\COP\HS_s)} -e^{-\beta \UL(x,\HS_s)} \right|\le Ce^{-\beta \BARUL(k(x),\COP\HS))}|\Delta_{q,N}(x,\sigma)|\le  CJ_L  \frac{q}{L} \|\nabla \VL \|_\infty
\COMMA
$$
for some  $C>0$.
\end{proof}

\section{Exact sampling of kMC dynamics}\label{exactsampling}

In addition to the approximate dynamics discussed so far, the ML-KMC method can also generate the {\it exact}  dynamics 
associated with the rates  $c(\sigma,\sigma')$ by appropriate choice of the coarse and reconstruction
rates, albeit at higher cost than the controlled-approximation dynamics $\HC(\sigma,\sigma')$.
More specifically, given $c(\sigma,\sigma')$ and $ \BARIT c(\eta,\eta')$,
$\crf(\sigma'|\eta',\sigma)$ can be  selected such that
\begin{equation}\label{eqRates}
    \BARIT{c}(\eta,\eta') \crf(\sigma'|\eta',\sigma)=\HC(\sigma,\sigma')\equiv c(\sigma,\sigma') \PERIOD
\end{equation}
In this case the ML-KMC method generates exactly the same stochastic process
with the direct kMC method achieving a perfect reconstruction. Relation
\VIZ{eqRates} ensures that processes $\{\HS_t \}_{t\ge 0}$ and $\{\sigma_t \}_{t\ge0}$
have the same generator $\HLl = \Ll$, %
which is sufficient to
prove that the two processes are identical, \cite{LIGGETT}.
As a specific example we demonstrate exact sampling via  the ML-KMC method for  the microscopic process  $\st$ generated  
by the Arrhenius rate for the adsorption-desorption mechanism in Section~\ref{microdynamics}.
For this model   the coarse space rate functions are explicitly
given by the coarse graining technique in Section~\ref{cgDynamics}. 
The   reconstruction rates that we present rely on 
{\it correcting} the error introduced by
coarsening at the null-event step.
Indeed, the rates on the coarse space 
$\CMspace$ corresponding to the compressed
interactions for $U(x,\sigma)$, \VIZ{microSL}, are given by
\begin{equation*}
  \bar{c}_a(k,\eta) = d_0 \left(Q-\eta(k)\right)\COMMA\;\;\; \bar{c}_d(k,\eta) = d_0\eta(k) e^{-\beta \bar{U}(k,\eta)}\COMMA
\end{equation*}
where 
$$
  \bar{U}(k,\eta) = \sum_{\substack{l\in \LATTC \\\ l\neq k}}[\BARIT K(k,l)+\BARIT J(k,l) ]\eta(l) + [\BARIT K(k,k)+\BARIT
     J(k,k)](\eta(k)-1)-\BARIT{h}(k)\PERIOD
$$
The reconstruction rates are explicitly defined by
\begin{equation}\label{ExactArrh} 
   \crf^a(x| k, \eta)=\frac{1-\sigma(x)}{Q-\eta(k)}\COMMA\;\;\;
   \crf^d(x| k, \eta)=\frac{\sigma(x)}{\eta(k)}e^{-\beta ( U(x,\sigma) -
   \bar{U}(k,\eta) )}\COMMA
\end{equation}
where we can also compare them to the reconstruction of the approximate dynamics in \VIZ{ApproxArrh}.
This choice of rates ensures  that the ML-KMC method generates the
same process $(\st, \Ll)$ since the two-level process has the rates
 $\HC(x,\sigma) =c(x,\sigma)$. Furthermore, 
 the quantities $U(x,\sigma) -
   \bar{U}(k,\eta)$ are better localized in the sense that they decay faster than $U(x, \sigma)$, 
hence it is easier to compress through truncation, \cite{AKPR}.
Finally, the detailed balance condition is 
satisfied with invariant measure $\mu_{N,\beta} (d\sigma)$; 
for completeness we give the proof in Appendix~\ref{AppendixDB}.

\medskip
\noindent
\begin{remark}
{\rm
In analogy to the path-wise decomposition \VIZ{MRates} of the
stochastic process, exact or controlled error equilibrium sampling has been
achieved with multilevel CGMC methods, \cite{KKPV}, based on an analogous
decomposition of the sampling probability measure 
$\mu(d\sigma)$, i.e., $\mu(d\sigma)=\bar\mu(d\eta)\nu(d\sigma|\eta)$, where $\nu$ defines
the reconstruction and $\bar\mu$ is the measure on the coarse space.
}
\end{remark}
\section{Acceleration and computational complexity}\label{Complexity}
The purpose of this section is to compare the efficiency  of a ML-KMC method with conventional methods.
 In the presence of   long-range interactions sampling with a
rejection-free algorithm is next to impossible due to the very high number of
classes in BKL-type methods, see for instance Table~\ref{Cost}. Hence we may inevitably be  forced to use a highly inefficient null
event algorithm. With the proposed ML-KMC  approach, a rather crude CG of the long-range potential gives 
rise to much fewer classes, thus we can sample at  the first level 
rejection-free using the BKL algorithm while the next level  can be null-event. 
 The ML-KMC algorithm is applied to the stochastic
lattice model for Arrhenius dynamics in Section~\ref{stlattice}, where we consider the global search implementation as
in the conventional stochastic simulation algorithm (SSA).
\begin{algorithm}\label{twoLevelkCGMC} Two-level ML-KMC \\
Given $\sigma, \eta=\COP\sigma$

\begin{description}
\item[Coarse level.] (Level I)
\mbox{}

\begin{description}
\item[Update.] 
(a) Calculate  transition rates $\BARIT c_a(k,\eta),\ \BARIT c_d(k,\eta)$, for all $k\in \LATTC$ and \\ $\BARIT\lambda^a_k(\eta)=\sum_{l<k} \BARIT c_a(l,\eta)$, 
$\BARIT\lambda^d_k(\eta)=\sum_{l<k} \BARIT c_d(l,\eta)$,
$\BARIT\lambda(\eta)= \BARIT\lambda_M^a(\eta)+\BARIT\lambda_M^d(\eta)$.
\item[Search.] Obtain  uniform random numbers $u_1, u_2\in [0,1)$.

If  $\BARIT\lambda_M^a(\eta) < u_1  $ adsorb  else desorb.
Assume that  adsorption is chosen, then 
find $k\in \LATTC$ such that $\BARIT\lambda^a_{k-1}(\eta) \le \BARIT\lambda_M^a(\eta) u_2 \le\BARIT\lambda^a_{k}(\eta)$.

\end{description}
\item[Microscopic level.] (Level II)
\mbox{}
\begin{description}
\item[Reconstruct.] Pick uniformly a site $x$ in the cell $C_k$.
\item[Accept/Reject.]  

Select a uniform random number $u\in [0,1)$, and define $ \crf(x|k, \eta)=\crf^{a,d}(x|k, \eta) $ 
according to the selected process in the coarse move, e.g. \VIZ{ApproxArrh} or \VIZ{ExactArrh}.
If $\lambda_{\trf}(\sigma) u \le \crf(x|k, \eta) $ accept and update the state at
the site $x$.
\end{description}
\item[Time update.] Update time from an exponential law  with  the parameter $\HL^*(\sigma)$, with $\HL^*(\sigma)=\BARIT\lambda(\eta)\lambda_{\trf}(\sigma) $.
\end{description}
\end{algorithm}
Here $\HL^*(\sigma)=\BARIT\lambda(\eta)\lambda_{\trf}(\sigma) $  is defined according to the sampling strategy. Specifically for the  exact sampling of Section~\ref{exactsampling}, 
\begin{eqnarray*}
     &\BARIT \lambda(\eta) = &\sum_k d_0(q-\eta(k)) + d_0\eta(k)e^{-\beta \BARIT U(k,\eta)}\COMMA  \\
     &\lambda_{\trf}(\sigma) =& q \max\left\{\frac{1}{q-\eta(k)}, \frac{1}{\eta(k)}e^{-\beta U^{*}}\right\} \COMMA 
\end{eqnarray*}
and for the approximating sampling of Section~\ref{ArrheniusMKMC},
\begin{eqnarray*} 
    & \BARIT \lambda(\eta)=   \sum_k d_0(q-\eta(k)) + d_0\eta(k)e^{-\beta \BARUL(k,\eta)}\COMMA  \\
    &\lambda_{\trf}(\sigma) = q \max\left\{\frac{1}{q-\eta(k)}, \frac{1}{\eta(k)}e^{-\beta U^{s*}}\right\} \COMMA  
\end{eqnarray*}
where $ U^*= \min_{x,\sigma}( U(x,\sigma) -\BARIT U(k,\eta))$ and  $ U^{s*}= \min_{x,\sigma}\US(x,\sigma) $.

The ML-KMC method provides an efficient balance between benefits and limitations
of the conventional null-event and rejection-free
methods, that we summarize in Table~\ref{Cost}.  
This  is achieved by (a) improving  the computational cost of a conventional rejection-free
method,  see    Table~\ref{table:3} and 
(b) increasing the successful events  of a  null-event method. 
An event is considered  successful when it is accepted and the system evolves to a new state.    
The cost per event of a   kMC algorithm can be divided in two categories. The {\it search} cost, the
computational cost to choose an event, and the cost of {\it updating} the rates when an event is performed.
In Table~\ref{Cost} the updating cost is realized as the number of operations
 necessary to calculate energy differences appearing and the search cost as the
length of the array from which the next event (site) is selected. We consider
global search and update algorithms for the comparison here, but we
mention that both the cost in the traditional rejection-free and ML-KMC
algorithms can be improved with the use of a sophisticated search/update algorithm, 
for example, with  binary tree methods, \cite{ACDV}. 
For the sake of comparison and completeness we describe  the conventional sampling algorithms, SSA, BKL and null event, 
in Appendix~\ref{SSABKL}.
The last column of  Table~\ref{Cost} reveals the acceleration of the method in generating successful events, for example,
when the system is at  the state $\sigma$ the rejection probability of a proposed  event in the ML-KMC algorithm is 
given by \VIZ{MLRejPr}.

Next we present another argument that reveals the fact that the proposed method improves the computational cost of  kMC algorithms.  
The number of classes in a BKL algorithm is determined by the level sets of $U(x, \sigma)$, in fact we can write 
$U(x, \sigma) : = \bar U(k, T\sigma)+ E(x,\sigma)$. Thus the level sets of $\bar U(k,\eta)$ are defined on a coarser
lattice, hence $U(x,\sigma)$ has many more ($q^d$ more) level sets. For a
potential decaying at the length $L$ on the microscopic lattice,  $U(x, \sigma)=\sum J(x-y) \sigma(y) = \BIGO(2^{dL})$ different 
values (and classes), while the coarse
interaction potential decays at a distance $L/q$ and $\bar U(k, T\sigma)=\sum_l\bar J(k, l) \eta(l)=\BIGO(q^{dL/q})=\BIGO(2^{dL \log(q)/q}) $ 
different values (and classes). Hence a BKL algorithm on the coarse space has, by the factor $2^{dL(1-\log(q)/q)}$,
less classes in the implementation of the BKL algorithm. 
Clearly  when the range of interactions $L$ is large, the number of classes
grows exponentially with $L$ and implementation of a microscopic BKL algorithm
is not feasible. Therefore sampling with a null-event algorithm is unavoidable.

\begin{table}[h]
\caption{Computational complexity  and  event rejection rate comparison for a single kMC step in one space dimension.}
  \begin{tabular}{lccc} 
\hline\noalign{\smallskip}
                      & Search          & Update                   & Rejection rate	\\
\noalign{\smallskip}\hline\noalign{\smallskip}
Rejection free (SSA)  & $\BIGO(N)$	& $\BIGO(L\times L)$	   &   0			 \\
Two-level ML-KMC (SSA) & $\BIGO( M )$    & $\BIGO(L/q\times L/q)$  & 1 - $\HL(\sigma)/\HL^*(\sigma)$ \\
Rejection free (BKL)  & $\BIGO( 2^L )$	& $\BIGO(L\times L)$	   &   0			 \\
Two-level ML-KMC(BKL)  & $\BIGO( 2^{L \log(q)/q})$ & $\BIGO(L/q\times L/q)$ & 1 -$\HL(\sigma)/ \HL^*(\sigma)$ \\
Null - event 	      & $ \BIGO(1)$     & $\BIGO( L) $             & 1- $\lambda(\sigma)/\lambda^*$ \\
\noalign{\smallskip}\hline\noalign{\smallskip}
 \end{tabular}
 \label{Cost}
 \centering
 \end{table}
The ML-KMC method achieves  acceleration of the rejection free simulations up of order $q$ when sampling for the same finite time interval. 
We also note that  since a transition to a new event is based on a single spin-flip,  the reconstruction  step is performed locally, 
confined  to a single coarse cell, a fact that improves further the computational cost of the method. 
The computational times (CPU) 
compared next are those needed for reaching the same real time $T$ with the conventional SSA and ML-KMC method, where we consider 
that CPU time is proportional 
to the computational complexity of the algorithms given in Table~\ref{Cost}. Let $n$ be the number of MC steps
necessary in a rejection-free method to reach real time $T$.
The corresponding necessary MC steps in a ML-KMC
 method are $m = n/\E_{\muN}\left[ p^{\multi}_{\SUCC}(\sigma)\right]$, where $ p^{\multi}_{\SUCC}(\sigma)=1-p^{\multi}_{\rej}(\sigma)$
is the   acceptance probability of an event when the system is in the state $\sigma$.
For the search algorithm the cost ratio of the microscopic SSA and the ML-KMC method   is
\begin{equation*}
    r_{s}=\frac{\CPU_{s,\mathrm{rej-free}}}{\CPU_{s,\multi}} \sim \frac{N n}{M m} = q\E_{\muN}[p^{\multi}_{\SUCC}]
\end{equation*}
and for  the update
\begin{equation*}
    r_{u}=\frac{\CPU_{u,\mathrm{rej-free}}}{CPU_{u,\multi}} \sim \frac{L^2 n}{ (L/q)^2 m} = q^2\E_{\muN}[p^{\multi}_{\SUCC}]\PERIOD
\end{equation*}

\section{Numerical experiments: an Ising-Curie-Weiss model}\label{sectionArrheniuskMC-Null}
We consider a benchmark problem with competing short and long-range
interactions that exhibits complex  multi-phase  behavior as captured in the phase diagrams 
in Figure~\ref{fig:2} and Figure~\ref{fig:9}.  
Exact solutions for the free energy in the thermodynamic
limit, $N\to \infty$, are known for the one-dimensional and two dimensional models, %
\cite{KAR}.
The energy of the system at the configuration $\sigma=\{\sigma(x), x\in\Lambda_N\} $ is defined
by the Hamiltonian
\begin{eqnarray}
   H_N(\sigma) &=& -\frac{K}{2}\sum_{x}\sum_{|x-y|=1}\sigma(x)\sigma(y)
                         -\frac{J}{2N}\sum_{x}\sum_{y\neq x} \sigma(x)\sigma(y)-h \sum \sigma(x) \label{Ising-Curie-Weiss}\\
                  &\equiv&  \HAMS(\sigma)+ \HAML(\sigma)+E(\sigma)\PERIOD \nonumber
\end{eqnarray}
The interactions involved in $\HAMS(\sigma)$ are   nearest-neighbor  with  the constant
strength $K$, while $\HAML(\sigma) $ represents the long-range interactions given by the potential $J$ with the range $L=N$
and $h$ is an external field. A closed form solution in the thermodynamic limit ($N\to \infty$) for the total coverage
$c_{\beta}(K,J,h)$ in the one-dimensional model was derived in \cite{KAR},
 \begin{equation}\label{cov}
       c_{\beta}(K,J,h) =\frac12  M_{\beta}(\frac14 K,\frac14 J,\frac12 h - \frac14 J - \frac14 K) +\frac12\COMMA
 \end{equation}
  where $M_{\beta}(K,J,h)$ is a solution (minimizer) of the  problem
 \begin{equation*} 
    \min_{m}\left( \!\frac{J}{2}m^2-\log(
         e^K\!\cosh(h+Jm)\!+\!(e^{2K}\!\sinh^2(h+Jm)\!+\!e^{-2K} )^{1/2})\!\right).
\end{equation*}
Depending on the system
parameters $c_{\beta}(K,J,h)$ can be a multivalued function and phase transitions may occur. 
We are interested in sampling the dynamical behavior of the system in the bi-stable regimes as well as in constructing the phase diagram
with respect to the external field $h$.

The computational examples demonstrate both acceleration of simulations with ML-KMC and the improved accuracy of ML-KMC contrasted 
with the CGMC simulations. The reference solution is obtained by the fully resolved microscopic simulation performed by the 
traditional null-event kMC. 
In numerical implementations we tested the three methods discussed previously:
\begin{description}
\item[{\rm (i)}  ] the direct null-event kMC, (Algorithm~\ref{null}),
\item[{\rm (ii)} ] the developed ML-KMC (Algorithm~\ref{twoLevelkCGMC}) sampling on the microscopic space,
\item[{\rm (iii)}] the null-event CGMC sampling on the coarse space only, i.e., with both short and long-range 
                   potentials coarse-grained and {\em without } corrections due to the reconstruction step in ML-KMC.
\end{description}
The energy difference $U(x,\sigma)$ appearing in the transition rates \VIZ{arrh} is
given by
\begin{equation*}
    U(x,\sigma)= K \sum_{|x-y|=1} \sigma(y) + J\sum_{y=1}^{N}\sigma(y) -h\PERIOD
\end{equation*}
For the ML-KMC method  we apply the potential splitting approach where the rates on the coarse space 
$\bar\Sigma_M$, at the first level of the method, are  defined  in \VIZ{cgdes} with the potential energy
\begin{equation*}
   \BARUL(k,\eta) = J\sum_{k=1}^M\eta(k) - \frac{h}{2} = J\sum_{y=1}^{N}\sigma(y) -\frac{h}{2}\COMMA
\end{equation*}
and the reconstruction rates at the second level of the method  are defined by
\VIZ{ApproxArrh}  with
\begin{equation*}
 \US(x,\sigma)=   K \sum_{|x-y|=1} \sigma(y) - \frac{h}{2}\PERIOD
\end{equation*}
To implement the null-event method we need a uniform upper bound of the rates
$\crf(x|k,\eta) $, \VIZ{ApproxArrh},  
\begin{equation}\label{lambdarfExample}
    \lambda_{\trf}(\sigma) =q\max\left\{\frac{1}{q-\eta(k)},\frac{1}{\eta(k)}\EXP{\beta (\frac{h}{2} +
           K_* )} \right\}
\end{equation}
where $\COP\sigma=\eta$ and $K_*=|\min\{0,K\}| $. 
Therefore the time step of the method, that is proportional to $\HL^*(\sigma)=\BARIT{\lambda}(\eta)\lambda_{\trf}(\sigma)$, 
clearly varies with the system state $ \sigma$
since   $\BARIT{\lambda}(\eta)=\sum_{k} [d_0 (q-\eta(k)) + d_0 \eta(k)e^{-\beta \bar{U}(k,\eta)}]$.

\medskip

Note that in this example the coarse-grained Hamiltonian $\BARHAML$ is exact, i.e.,
$\BARHAML(\eta) \equiv \HAML(\sigma)$, thus there is no approximation error due to coarse-graining the long-range
potential.
Therefore, while CGMC sampling is approximate, due to coarse-graining of both $\BARHAML(\eta) $ and $\BARHAMS(\eta) $, 
the ML-KMC method samples the exact microscopic process, i.e., $ \HC(x,\sigma)= c(x,\sigma)$ for all $x\in \LATT,\sigma\in\Sigma_N$.
This allows us to quantify the effect of splitting the potential function into short and long-range parts. 
For example,  Figure~\ref{fig:4} shows that the potential splitting is not introducing errors, 
which verifies the theoretical estimate for the information loss of the equilibrium distribution, Theorem~\ref{StatEstimate}.
The effect of the splitting is apparent only in the average acceptance rate of the method where the strength 
of the short-range interactions $K$  controls the rejection rate according to \VIZ{lambdarfExample}. 

\smallskip

In order to test the effect of coarse-graining in the ML-KMC method we modify the long-range
potential in the Hamiltonian \VIZ{Ising-Curie-Weiss} and
consider finite-range interactions, with the range $L<N$, and with the long-range part of the Hamiltonian
\begin{equation*}
    \HAML(\sigma)=\frac{J}{2L}\sum_{x}\sum_{|y- x|\le L} \sigma(x)\sigma(y)\PERIOD
\end{equation*}
For this case the proposed ML-KMC method is approximate, however, it still reduces significantly
the coarse-graining error of the direct CGMC sampling, since compressing of the
short-range part is avoided, see Figure~\ref{fig:5} and description below.

In all simulations we consider the one-dimensional model with the coarsening
parameter $q=N$ in the CGMC and ML-KMC methods, that is the coarse
space consists of one cell $k=1$ and the coarse variable $\eta$ is the total
coverage $\eta=\COP\sigma=\sum_{x\in \LATT} \sigma(x)$.

\medskip
\noindent
{\it Stationary dynamics and equilibrium sampling.} We demonstrate properties of the ML-KMC algorithm 
in the stationary regime by constructing (equilibrium) phase diagrams of the average coverage with respect to the 
external field $h$. We explore different regimes of the phase plane $K$-$J$. With the choice of the
potential parameters $K$ and $J$ corresponding to bi-stable regimes we observe that the ML-KMC algorithm
approximates properly the hysteresis behavior while the CGMC algorithm samples incorrect energy landscape 
and thus does not estimate the hysteresis behavior correctly. The coarse-graining parameter is set to $q=N$
both in ML-KMC and CGMC simulations. The ML-KMC method
avoids compressing the short-range interactions that introduce large error in the CGMC simulations. 
Figure~\ref{fig:2} depicts the hysteresis behavior in the case when the long-range potential is of
Curie-Weiss type, i.e., the interaction range is $L=N$, and hence it is coarse-grained exactly by block-spin 
coarse variables. However, coarse-graining the short-range, nearest-neighbor Ising, potential introduces
an error which leads to a wrong prediction of hysteresis in the CGMC simulation.
In Figure~\ref{fig:8} and \ref{fig:9} we chose the interaction range $L<N$ which also introduces coarse-graining
error in the coarse-grained long-range potential. However, the presented error analysis for the invariant measure
suggests that this error is small and thus the ML-KMC sampling, unlike the CGMC simulations, are in a good
agreement with estimates from the microscopic simulations. 

\begin{figure}[ht]
\centering
\includegraphics[angle=-0,width=0.8\textwidth,height=0.4\textheight]{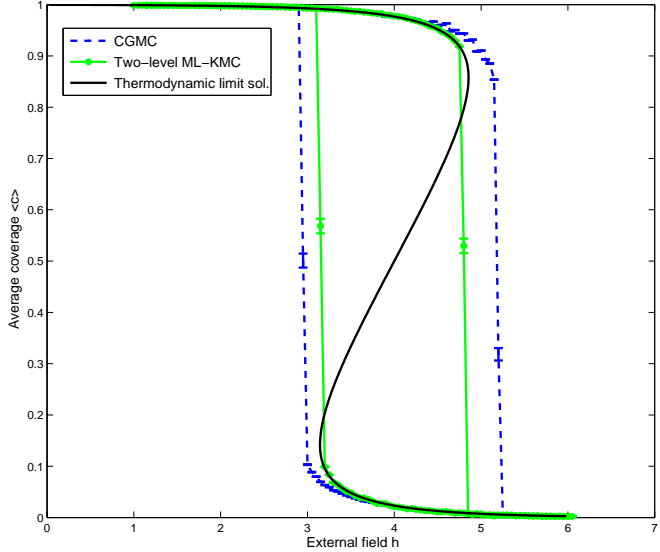}
\caption{Hysteresis simulation in a bi-stable regime. 
         Potential  parameters $K =3$, $J =5$,  $L=N$, and the lattice-size $N =1024$
         and the coarsening parameter $q=N$.
        }
\label{fig:2}
\end{figure}

\begin{figure}[ht]
\centering
\includegraphics[angle=-0,width=0.8\textwidth,height=0.4\textheight]{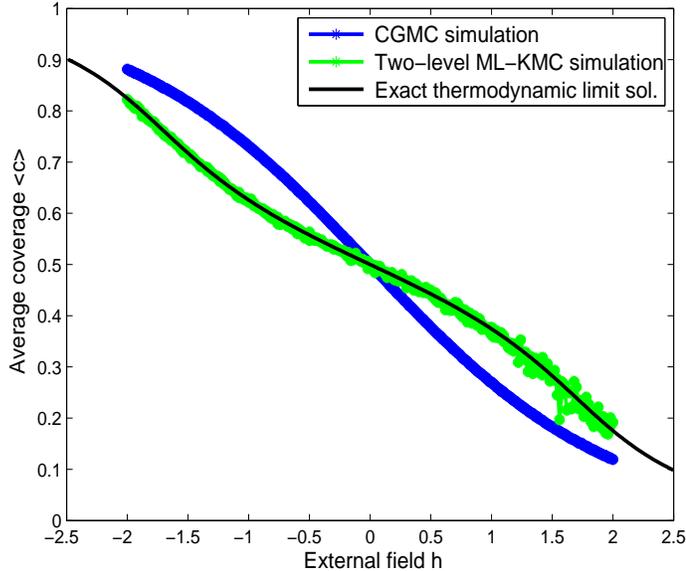}
  \caption{Hysteresis simulation in a single phase regime. 
          Potential parameters $K =-5$, $J =5$, $L=20$, the lattice size $N =256$,  and
          the coarsening parameter $q=N$.}
\label{fig:8}
\end{figure}

\begin{figure}[ht]
\centering
  \includegraphics[angle=-0,width=0.8\textwidth,height=0.4\textheight]{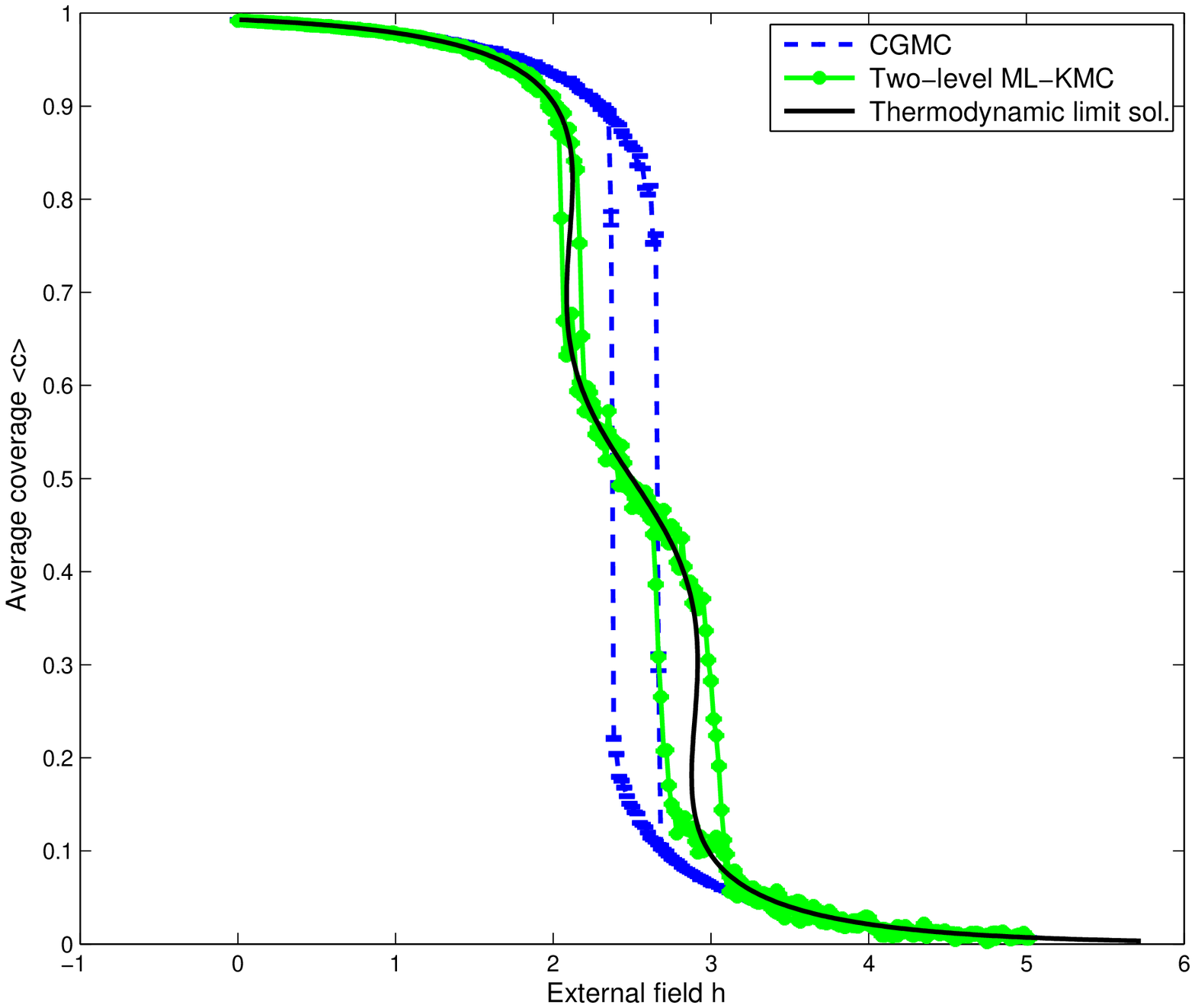}
  \caption{ Hysteresis simulation in a bi-stable regime.  
    Potential parameters $K =-5$, $J =10$, $L=100$, the lattice size $N =1024$, and the coarsening parameter $q=N$. }
\label{fig:9}
\end{figure}

 \medskip
\noindent
{\it Transients and dynamical sampling.} 
In the dynamical sampling we explore two quantities  of  interest:

\begin{description}
\item[{\rm (a)}] The {\it path-wise} behavior of the coverage, defined by $c(\sigma_t)
=\frac{1}{N}\sum_{x\in\LATT}\sigma_t(x)$ for the direct microscopic sampling,
$\HC (\HS_t) =\frac{1}{N}\sum_{x\in\LATT}\HS_t(x) $ for the two-level sampling
and $c (\eta_t) =\frac{1}{N}\sum_{k\in\LATTC}\eta_t(k) $ for the CGMC sampling.

\item[{\rm (b)}] The mean time to reach a transition from one equilibrium to another in the bi-stable regime, the {\it exit time}, 
$\tau= \EXPECT[T]$, $T=\inf\{t>0: c_t\ge C\}$. The probability
density functions (PDFs) $\rho_m $, $\rho_{\mathrm{tl}} $, $\rho_{\mathrm{cg}} $ for the exit time
estimated in the microscopic, the two-level ML-KMC and the CGMC methods
respectively are monitored. Starting from an initial state with the coverage $c_0=0$ in
all methods we record the time $\tau$ when the coverage exceeds the value $C=0.99$.
\end{description}

Estimating the observable $\tau$ tests both a proper approximation of the energy landscape as well as the correct
time-scale in approximating dynamics. The simulation is set for the parameters $K$, $J$ such that the system exhibits transition 
to an equilibrium which is, depending on the value of the external field $h$, stable or metastable (see Figure~\ref{fig:2}). 
We compare not only the expected (mean) values but also the probability density functions (PDFs) in order to demonstrate 
importance of error estimates in terms of the relative entropy. Probability density function was estimated 
from $10^4$ independent samples using the MATLAB estimator {\tt ksdensity} with a normal kernel function. 

Figure~\ref{fig:1} shows the comparison in the case of a single equilibrium state $c\approx 1.0$ (for the given value of $h$)
and the long-range potential which is coarse-grained exactly, i.e., $L=N$. We observe a good agreement in
all three methods, although the CGMC introduces a visible error due to coarse-graining of the short-range potential.

\begin{figure}[ht]
\centering
\includegraphics[angle=-0,width=0.8\textwidth,height=0.4\textheight]{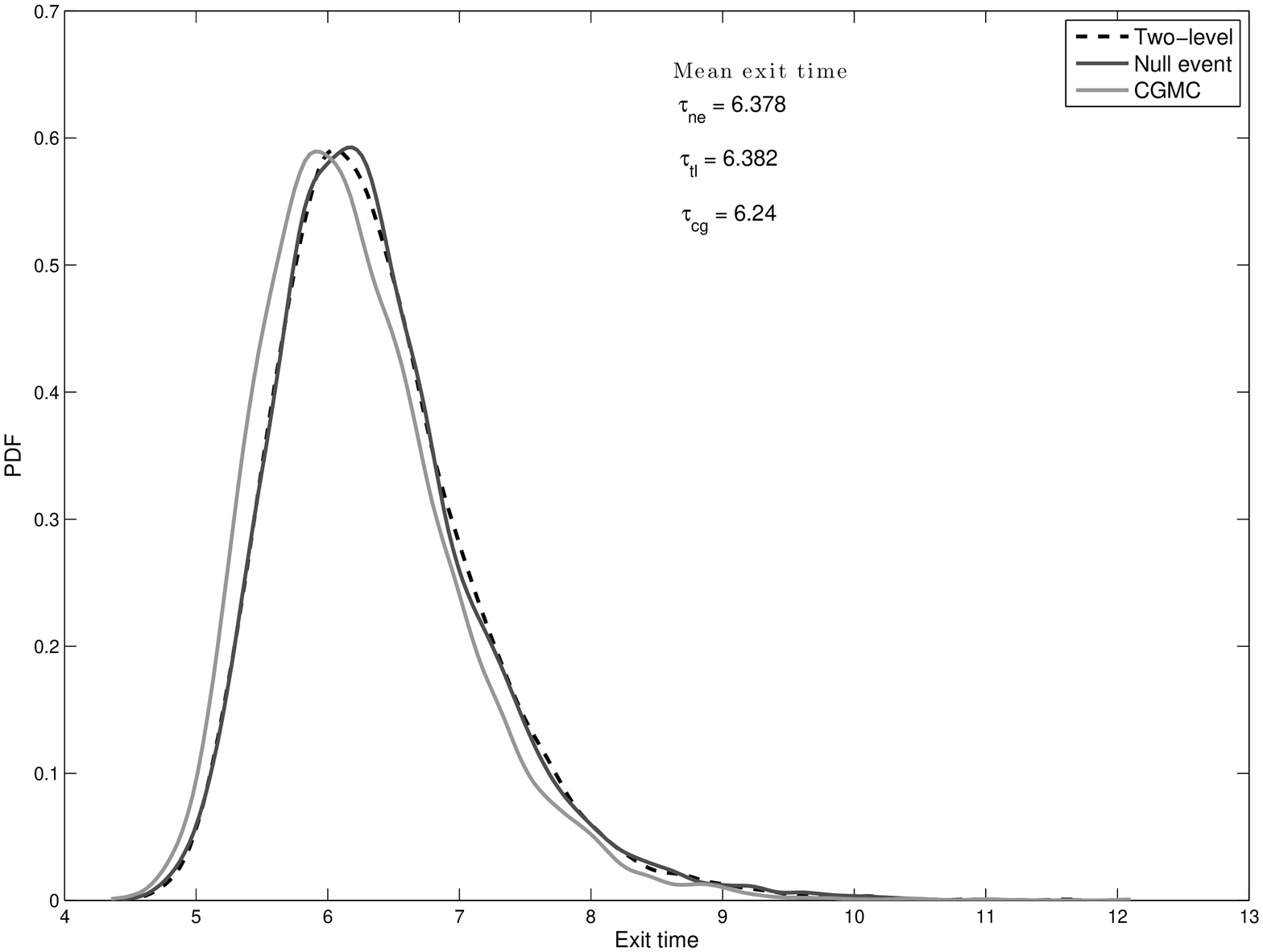}
   \caption{Comparing the probability density function of the exit time from the initial coverage $c_0=0$ to $c\geq 0.99$. 
            Potential parameters $K =2$, $J =5$, $h=2$, $L=N$, the lattice size $N =1024$
            and the coarse-graining parameter $q=N$.
            }
\label{fig:1}
\end{figure}

An additional error is introduced by coarse-graining the long-range potential with $L=100<N$. Comparison
of the exit time PDF in the case of a single equilibrium state $c\approx 1.0$ (for the given value of $h$)
is depicted  Figure~\ref{fig:6}. While ML-KMC simulations are in a good agreement with the microscopic simulation the
CGMC algorithm introduces significant error for the estimated PDF.

\begin{figure}[ht]
\centering
 \includegraphics[angle=-0,width=0.8\textwidth,height=0.4\textheight]{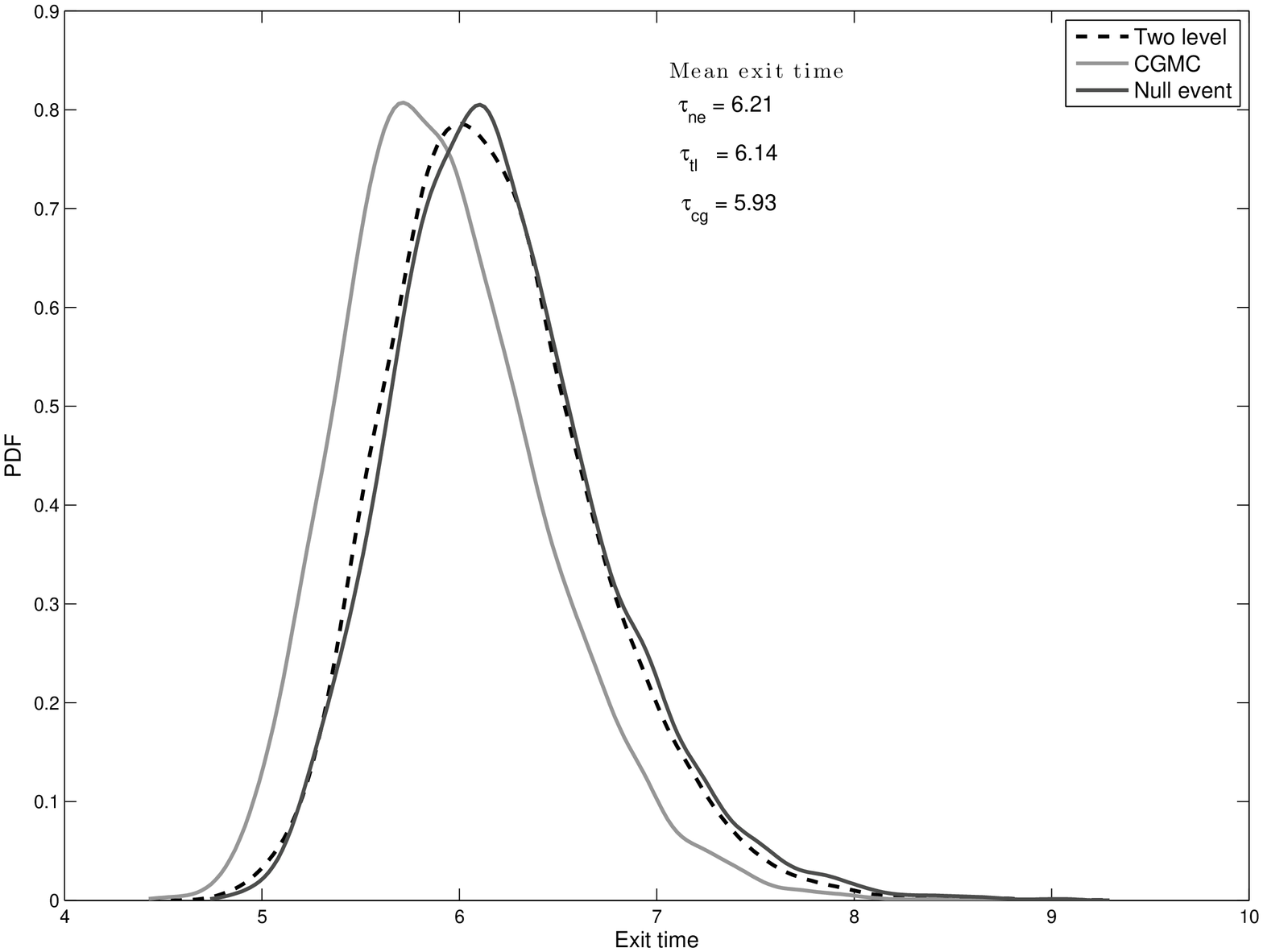}
\caption{Comparing the probability density function of the exit time from the initial coverage $c_0=0$ to $c\geq 0.99$. 
         Potential parameters $K =3$, $J =5$, $h=2.5$, $L=100$, the lattice size $N=1024$, and the coarse-graining parameter $q=N$.
        }
\label{fig:6}
\end{figure}

By adjusting the external field $h$ the sampling is performed in the bi-stable regime with two meta-stable 
equilibria. Coarse-graining both short and long-range potentials changes significantly the energy landscape and
the CGMC algorithm cannot capture the transition within the simulation time-window. The exit-time probability
distribution function is depicted in Figure~\ref{fig:5} showing that the ML-KMC algorithm is capable
of capturing the transition and approximate the exit-time PDF. The inset demonstrates that the CGMC
simulation was unable to estimate the mean exit time as no transition occurred and
the exit-time PDF is concentrated at the final time of the simulation window. This fact
is further visualized in Figure~\ref{fig:4} where evolution of the mean coverage is depicted. While 
the ML-KMC simulation results in a trajectory that approximates well the reference trajectory obtained
from microscopic null-event kMC with a transition from $c=0$ to $c\approx 1$ equilibrium, the trajectory
averaged in the CGMC simulation does not exhibit any transition in the simulation window.

\begin{figure}[ht]
\centering
 \includegraphics[angle=-0,width=0.8\textwidth,height=0.4\textheight]{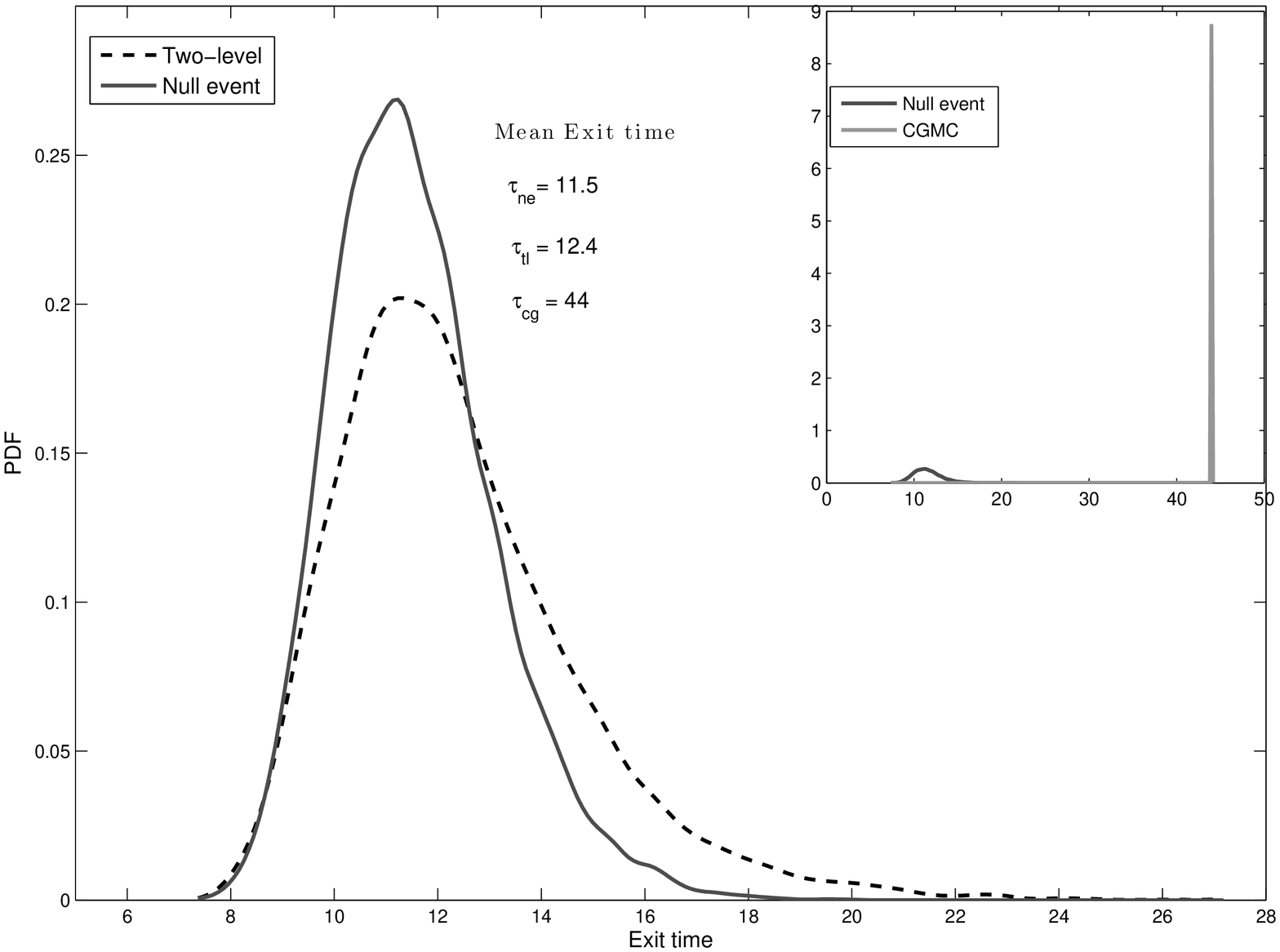}
\caption{Comparing the probability density function of the exit time from the initial coverage $c_0=0$ to 
         $c\ge 0.99$.  
         Potential parameters $K =3$, $J =5$, $h=3.1$, $L=100$, the lattice size $N=1024$, and the coarse-graining parameter $q=N$.
         Coarse-graining  error of CGMC that appears due to the finite-range interactions is substantially reduced 
         with the ML-KMC method (see the text for explanation of the inset).
         }
\label{fig:5}
\end{figure}

\begin{figure}[ht]
\centering
 \includegraphics[angle=-0,width=0.8\textwidth,height=0.4\textheight]{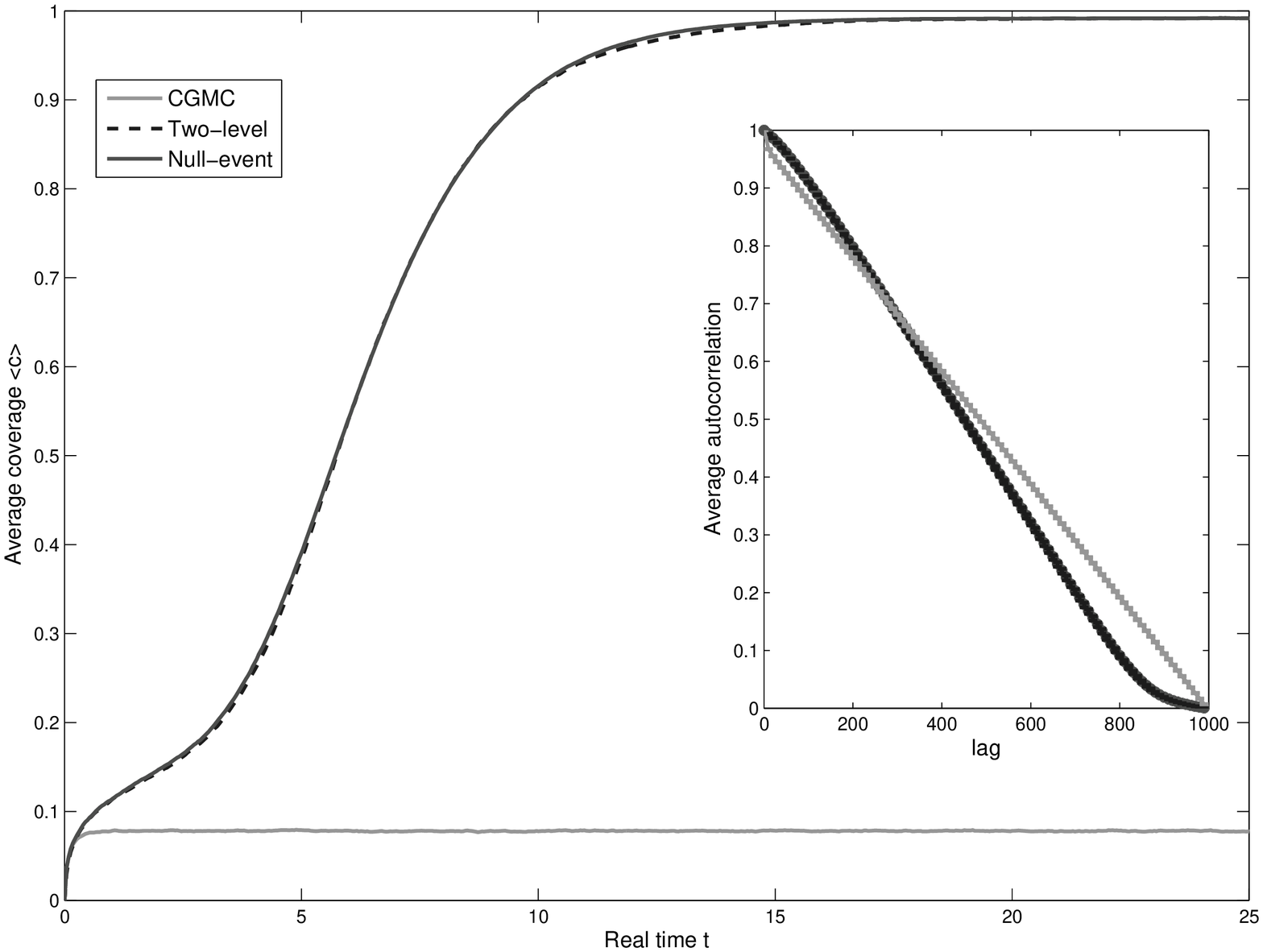}
\caption{Average coverage trajectory. As external field $h$ approaches critical value $h_c=4$ the coarse-graining error in CGMC becomes important. 
         However, the two-level ML-KMC simulations capture correctly transitions even when CGMC fails. 
         Potential parameters $K =3$, $J =5$, $h=3.1$, $L=N$, the lattice size $N=1024$, and the coarse-graining parameter $q=N$.
         The inset depicts the estimated autocorrelation function. 
         }
\label{fig:4}
\end{figure}

In Table~\ref{table:2} we compare numerical results for the exit time and the corresponding computational times 
of the three algorithms for different values of the potential parameters.  
For the finite range $L<N$ interactions, where coarse-graining error is present,  
we see that the ML-KMC method estimates are closer 
to the microscopic (conventional) method even  when the CGMC method fails.   
Furthermore, we see a significant acceleration of the computational time both with the CGMC 
and the ML-KMC method.   

\begin{table}[ht]
\caption{Approximation of the exit time $\tau$.   
For the statistics  we use $10^4$ samples and present the $95\% $ confidence interval.
The  potential parameter $J=5$, the coarse-graining parameter $q=N$ and the lattice size $N=1024$ are fixed. }
\centering  
\begin{tabular}{lc c  c c c c }  
\hline\noalign{\smallskip} 
Parameters & $\tau_m$       & $\tau_{\mathrm{tl}}$ & $\tau_{\mathrm{cg}}$  & $\CPU_m$   & $\CPU_{\mathrm{tl}}$& $\CPU_{\mathrm{cg}}$\\
           & microscopic    & ML-KMC               & CGMC                  &  [sec]     &    [sec]            &   [sec]              \\
 [0.5ex] \noalign{\smallskip}\hline\noalign{\smallskip}
$L=N $     &                &                      &                       &            &                     &     \\
$K=0,h=1$  & 28.5 $\pm$0.8  & 28.3$\pm$0.8         & 28.7$\pm$0.8          & 1534       & 9                   &8    \\
$K=2,h=2$  & 6.40$\pm$0.03  & 6.40$\pm$0.03        & 6.20$\pm$0.02         & 884        & 6                   &5    \\
$L=100 $   &                &                      &                       &            &                     &     \\
$K=3,h=2.5$& 6.20$\pm$0.02  & 6.1$\pm$0.03         & 5.93$\pm$0.02         & 158        & 9                   & 7   \\
$K=3,h=3.1$& 11.50$\pm$0.06 & 12.4$\pm$0.1         & 44.0$\pm$0.1          & 526        & 45                  & 100 \\
\hline  
\end{tabular}
\label{table:2}
\end{table}

\begin{table}[ht]
\caption{CPU time (seconds): The evolution final time $T=20$, the potential parameters $K=1$, $J=5$, $h=2.5$, $L =N$, and
         the coarse-graining parameter $q = N$}
\centering  
\begin{tabular}{lc c  }  
\hline\noalign{\smallskip} 
Lattice size $N$      &  Null event  &   ML-KMC \\
 [0.5ex] \noalign{\smallskip}\hline\noalign{\smallskip}
        512           &     9        &    0.5   \\
        1024          &    33        &    0.9   \\ 
        2048          &   131        &    1.7   \\
        4096          &   514        &    4     \\
        8192          &   2143       &    13    \\
\hline  
\end{tabular}
\label{table:3}
\end{table}

\appendix

\section{Detailed balance}\label{AppendixDB}
\subsection{Exact dynamics} 
The rate $\HC(x,\sigma) $, \VIZ{ExactArrh}, satisfies the detailed balance condition with
$\mu_{N,\beta} (d\sigma)$, \VIZ{gibbs}, i.e.,
\begin{eqnarray*}
   \HC(x,\sigma) e^{-\beta H_N(\sigma)} = \HC(x,\sigma^x) e^{-\beta H_N(\sigma^x)}\COMMA
\end{eqnarray*}
since $\HC(x,\sigma) =c(x,\sigma)$ for all $\sigma \in \Sigma_N$ and $x\in\LATT$, 
and $c(x,\sigma)$ satisfies \VIZ{DB}. We have $\HC(x,\sigma) =c(x,\sigma)$  since
\begin{eqnarray*}  
   \HC_a(x,\sigma)= \bar{c}_a(k,\eta) \crf^a(x|k,\eta) &=&d_0 (q-\eta(k))\frac{1-\sigma(x)}{q-\eta(k) } \\
                      &=&d_0 (1-\sigma(x)) = c_a(x,\sigma)  
\end{eqnarray*}
and
\begin{eqnarray*} \label{equalCd}
\HC_d(x,\sigma)= \bar{c}_d(k,\eta)   \crf^d(x|k,\eta)   
&=&d_0 \eta(k)e^{-\beta \bar{U}(k,\eta)} \frac{ \sigma(x)}{ \eta(k) }
e^{-\beta[U(x,\sigma)-\bar{U}(k,\eta)]} \\
  &=&d_0  \sigma(x) e^{-\beta U(x,\sigma)}= c_d(x,\sigma) \nonumber \PERIOD
\end{eqnarray*}
\subsection{Approximate dynamics} 
The approximate reaction rates $\HC(x,\sigma)$  defined in \VIZ{ApprRates}
satisfy the DB condition with invariant measure $\widetilde\mu_{N, \beta}(d \sigma)$ \VIZ{approxGibbs}.
Indeed, if  we denote $\HC_a(x,\sigma)= \bar{c}_a(k,\eta) \crf^a(x|k,\eta) $ 
and $\HC_d(x,\sigma)=  \bar{c}_d(k,\eta)   \crf^d(x|k,\eta)  $ we have
\begin{eqnarray*}
   \HC_a(x,\sigma) e^{-\beta \widetilde H_N(\sigma)} &&= \left[ \bar{c}_a(k,\eta)
   \crf^a(x|\eta) \right] e^{-\beta \widetilde H_N(\sigma)} \\
   &&= \left[ d_0(q-\eta(k))\frac{1-\sigma(x)}{q-\eta(k)} \right] e^{-\beta (\widetilde
   H_N(\sigma^x)-(2\sigma(x) -1)\HU(x,\sigma) )}\\
   &&= \left[ d_0 (1-\sigma(x)) e^{-\beta \HU(x,\sigma)} \right] e^{-\beta \widetilde H_N(\sigma^x)} \\
   &&= \left[ d_0 \sigma^x(x)e^{-\beta \HU(x,\sigma)} \right] e^{-\beta \widetilde H_N(\sigma^x)} \\
   &&= \HC_d(x,\sigma^x)   e^{-\beta \widetilde H_N(\sigma^x)} \COMMA
\end{eqnarray*}
and similarly
\begin{eqnarray*}
   \HC_d(x,\sigma) e^{-\beta \widetilde H_N(\sigma)} &&= \left[ \bar{c}_d(k,\eta)
   \crf^d(x|\eta) \right] e^{-\beta \widetilde H_N(\sigma)} \\
   &&= \left[ d_0 \eta(k) e^{-\beta \bar{U}_l(k,\eta)} \frac{ \sigma(x)}{ \eta(k)}
      e^{-\beta \US(x,\sigma) } \right] e^{-\beta (\widetilde H_N(\sigma^x)-(2\sigma(x)-1)\HU(x,\sigma) )}\\
   &&= \left[ d_0 \sigma(x) e^{-\beta (\bar{U}_l(k,\eta) + \US(x,\sigma))} e^{\beta\HU(x,\sigma) } \right] 
      e^{-\beta \widetilde H_N(\sigma^x)} \\
   &&= \left[  d_0 (1-\sigma^x(x) ) \right]    e^{-\beta \widetilde H_N(\sigma^x)}  \\
   &&= \HC_a(x,\sigma^x)    e^{-\beta \widetilde H_N(\sigma^x)}\PERIOD
\end{eqnarray*}
\section{Proof of Lemma~\ref{lemU}}\label{ProofOfLemma}
For the sake of completeness we also give the proof of Lemma~\ref{lemU} which was proved in \cite{KPS}.
\begin{proof}
We denote by $\nabla_{\sigma} \phi(\sigma) =
\left(\partial_x\phi(\sigma)\right)_{x\in \LATT}$ and $ c(\sigma) = \left( c(x,\sigma)\right)_{x\in\LATT} $.
The equation \VIZ{pdeU} can be rewritten as 
\begin{equation*}
 \partial_t u(t,\sigma)+ c(\sigma)\cdot \nabla_{\sigma} u(t,\sigma) =0\PERIOD
\end{equation*}
For the discrete difference $\partial_x u$ we obtain the equation
\begin{equation*} 
  \partial_t \left(\partial_x u(t,\sigma)\right) + c(\sigma)\cdot
  \nabla_{\sigma} \partial_x u(t,\sigma) + \partial_x c(\sigma)\cdot
  \nabla_{\sigma} \partial_x u(t,\sigma^x) =0\PERIOD
\end{equation*}
From the definition of the rates $ c(x,\sigma)$, \VIZ{arrh}, we can
 estimate upper bounds for $\partial_x c(\sigma)$. Using that
$ \partial_x U(z,\sigma) = U(z,\sigma^x) -U(z,\sigma) =
K(x-z)(1-2\sigma(x)) + J(x-z)(1-2\sigma(x)) $ when $z\neq x $, and
$\partial_x U(z,\sigma)=0$ when $z=x$ we have
\begin{equation*}
 \partial_x c(z,\sigma) =
\begin{cases}
\BIGO(1), \;  &\textrm{ for } z=x\COMMA\\
   \BIGO(1)  , \; &\textrm{ for } 0< |z-x|\le S\COMMA \\
   \BIGO(1/L), \;  & \textrm{ for }  S < |z-x|\le L\PERIOD \\
\end{cases}
\end{equation*}
Then, since $ \Ll v(\sigma)=c(\sigma)\cdot \nabla_{\sigma} v(\sigma)$,
we can write
\begin{eqnarray*}
&& \partial_t \left(\partial_x u(t,\sigma)\right) + \Ll \partial_x
     u(t,\sigma) + \sum_{z\in \LATT}\partial_x c(z,\sigma) \partial_z
     u(t,\sigma^x) =0\COMMA \\
&& \partial_t \left(\partial_x u(t,\sigma)\right) + \Ll \partial_x u(t,\sigma) +
     \BIGO(1)\partial_x u(t,\sigma^x)\\
&&\;\; +\BIGO(1) \sum_{|z-x|\le S}  \partial_z u(t,\sigma^x) +
         \BIGO(\frac{1}{L}) \sum_{S<|z-x|\le L}  \partial_z u(t,\sigma^x)=0\PERIOD
\end{eqnarray*}
Furthermore, we have 
\begin{multline}
\| \partial_x u(t,\cdot) \|_{\infty} \le \| \partial_x u(0,\cdot) \|_{\infty}
+\int_{t}^T \BIGO(1) \| \partial_x u(s,\cdot) \|_{\infty} \,ds + \\
+\int_{t}^T \BIGO(1) \sum_{|z-x|\le S} \| \partial_x u(s,\cdot) \|_{\infty}\, ds +
\int_{t}^T \BIGO(\frac{1}{L}) \sum_{S<|z-x|\le L} \| \partial_x u(s,\cdot)\|_{\infty}\, ds\PERIOD
\end{multline}
Based on this relation, application of Gronwall's inequality for
$\theta(t)=\sum_x\| \partial_x u(t,\cdot) \|_{\infty} $ and the fact that $S$ is
finite and small we conclude
$$
\sum_x\| \partial_x u(t,\cdot) \|_{\infty}\le e^{c(T-t)}\sum_x\| \partial_x
u(0,\cdot) \|_{\infty} \PERIOD
$$
\end{proof}

\section{Kinetic Monte Carlo algorithms}\label{SSABKL}
Stochastic simulation algorithm (SSA) as proposed in \cite{GILLESPIE} is
described next, where the evolution from a state $\sigma $ to $\sigma^x$ is
sampled by:
\begin{algorithm}\label{SSA} Stochastic Simulation Algorithm. \\
\begin{description}
\item[Step 1: Update.]
       (a) Calculate all rates $c(y,\sigma), \forall y\in \Lambda_N $ form \VIZ{arrh},
          that are affected from the previews event.\\
       (b) Calculate $\lambda_x(\sigma)=\sum_{y<x} c(y,\sigma)$ and
$\lambda(\sigma)=\sum_{y\in \Lambda_N} c(y,\sigma)$
\item[Step 2: Search.] Obtain a uniform random number $u\in [0,1)$ and
      search for $x \in \Lambda_N$ such that
      $$ 
           \lambda_{x-1} (\sigma) < \lambda(\sigma)  u \le \lambda_x(\sigma) 
      $$  
\item[Step 3] Update time from an exponential law  with  the parameter $\lambda(\sigma)$ or 
      equivalently with the mean $\Delta t = \frac{1}{\lambda(\sigma)}$.
      That is select  a uniform random number $u_1\in [0,1)$ and update the
      time as $t'= t + \delta t$ with   $\delta t = - \log(u_1)\Delta t$.
\end{description}
\end{algorithm}
\medskip
\noindent
The $n$-fold way (or BKL) algorithm, \cite{BKL}, is equivalent to the SSA in the
sense that it always leads to a successful event and requires the updating of
all transition rates. The BKL algorithm was already designed to reduce the cost of the searching process by dividing 
the transition states into classes (the number of classes $n\ll N$) with
the same probability, thus the search algorithm cost depends on the number of
processes at each site, i.e., scales linearly with the reaction range 
(the number of interacting neighbours) in the model under consideration.

\begin{algorithm}\label{BKL} $n$-fold algorithm (BKL). \\
Given $\sigma$
\begin{description}
\item[Step 1: Update.]
   (a) Calculate all rates $c(y,\sigma), \forall y\in \Lambda_N $ that are affected by the previous event.
\item[Step 2: Search.]
   Group sites $x\in \LATT$ in classes $D_i$, $i=1,\dots,n$ by classifying them with their rate values and define 
   $$
      Q_j(\sigma) = \sum_{i=1}^j\sum_{y\in D_i} c(y,\sigma) = \sum_{i=1}^j |D_i| c(y,\sigma)
   $$ 
   for some $y\in D_i$, $j = 1,\dots,n$. Generate a uniform random number $u\in [0,1)$ and search for $i = 1,\dots, n $such that
   \begin{equation*}
       Q_{i-1} (\sigma)< Q_{n}(\sigma) u \le  Q_{i}(\sigma)\COMMA
   \end{equation*}
    then choose $x\in D_i$ uniformly.
\item[Step 3] Update time from the exponential law  with  the parameter $\lambda(\sigma) = Q_n(\sigma)$, or equivalently with
    the mean $\Delta t = \frac{1}{\lambda(\sigma)}$. %
\end{description}
\end{algorithm}

\medskip
\noindent
The previous two algorithms are in the class of rejection-free methods as the embeded Markov chain always jumps into a new
state. However, by applying the uniformization we obtaine a null-event algorithm in which the embeded chain has nonzero probability
to stay at the same state in each step.
\begin{algorithm}\label{null} Null-event algorithm. \\
Find the bounds $\lambda^{*,\mathrm{loc}}= d_0\max\{1, e^{-\beta U^*}\}$, and $U^* =\min_{x,\sigma} U(x,\sigma) $.\\
Given $\sigma$
\begin{description}
\item[Step 1: Search/Update.]  Select a site $x \in \LATT$ with the uniform probability $\frac{1}{N}$ and calculate $c(x,\sigma)$.
\item[Step 2:Accept/Reject.] Obtain a uniform random number $u\in [0,1)$,\\ 
    if $c(x,\sigma) \geq \lambda^{*,\mathrm{loc}} u$  accept  and update the state $\sigma \to \sigma^x$ at the site $x$, \\
    if $c(x,\sigma) < \lambda^{*,loc} u$ assign the new state to be $\sigma$.
\item[Step 3]  Update time from the exponential law  with  the parameter $\lambda^{*,\mathrm{loc}}$, or equivalently
    with the mean $\Delta t = \frac{1}{\lambda^{*,\mathrm{loc}} }$.
 \end{description}
\end{algorithm}

\bibliographystyle{siam}

\end{document}